\documentclass[11pt]{article}
\usepackage{blindtext}
\usepackage{enumitem} 
\setlist[enumerate]{parsep=0pt}
 
\usepackage{ragged2e} 

\usepackage{mathrsfs}
\usepackage{amssymb}
\usepackage{layout}
\usepackage{graphicx}
\usepackage{marvosym}
\usepackage[usenames,dvipsnames]{color}
\usepackage{verbatim}
\usepackage{enumitem}
\usepackage[pdftex]{hyperref}
\usepackage{fancyhdr}
\usepackage{indentfirst}
\usepackage{centernot}
\usepackage[all]{xy}

\usepackage{amsthm,etoolbox}

\makeatletter
\@addtoreset{theorem}{section}
\@addtoreset{theorem}{subsection}
\@addtoreset{theorem}{subsubsection}

\usepackage{amsmath}

\usepackage{tikz-cd}
\usetikzlibrary{calc}

\usepackage{circuitikz}

\usepackage{hyperref}
\usepackage{cleveref}[noabbrev]
\newtheorem{theorem}{Theorem}

\theoremstyle{definition}

\newtheorem{definition}[theorem]{Definition}

\theoremstyle{theorem}

\newtheorem{lemma}[theorem]{Lemma}
\newtheorem{prop}[theorem]{Proposition}
\newtheorem{cor}[theorem]{Corollary}

\newtheorem{remark}[theorem]{Remark}

\crefname{theorem}{Theorem}{Theorems}
\crefname{lemma}{Lemma}{Lemmas}
\crefname{prop}{Proposition}{Propositions}
\crefname{fact}{Fact}{Facts}
\crefname{remark}{Remark}{Remarks}
\crefname{cor}{Corollary}{Corollaries}

\newcommand{\theoremprefix}{}
\let\thetheoremsaved\thetheorem
\renewcommand{\thetheorem}{\theoremprefix\thetheoremsaved}

\patchcmd{\@startsection}{\par}{\renewcommand{\theoremprefix}{\csname the#1\endcsname.}}{}{}

\makeatother

\begin{document}

\clearpage
\pagenumbering{arabic}

\def\tp{\mbox{\rm tp}}
\def\qftp{\mbox{\rm qftp}}
\def\cb{\mbox{\rm cb}}
\def\wcb{\mbox{\rm wcb}}
\def\Diag{\mbox{\rm Diag}}
\def\trdeg{\mbox{\rm trdeg}}
\def\Gal{\mbox{\rm Gal}}
\def\Lin{\mbox{\rm Lin}}

\def\restriction#1#2{\mathchoice
              {\setbox1\hbox{${\displaystyle #1}_{\scriptstyle #2}$}
              \restrictionaux{#1}{#2}}
              {\setbox1\hbox{${\textstyle #1}_{\scriptstyle #2}$}
              \restrictionaux{#1}{#2}}
              {\setbox1\hbox{${\scriptstyle #1}_{\scriptscriptstyle #2}$}
              \restrictionaux{#1}{#2}}
              {\setbox1\hbox{${\scriptscriptstyle #1}_{\scriptscriptstyle #2}$}
              \restrictionaux{#1}{#2}}}
\def\restrictionaux#1#2{{#1\,\smash{\vrule height .8\ht1 depth .85\dp1}}_{\,#2}} 

\newcommand{\forkindep}[1][]{%
  \mathrel{
    \mathop{
      \vcenter{
        \hbox{\oalign{\noalign{\kern-.3ex}\hfil$\vert$\hfil\cr
              \noalign{\kern-.7ex}
              $\smile$\cr\noalign{\kern-.3ex}}}
      }
    }\displaylimits_{#1}
  }
}
\newpage

\begin{center}

\large \MakeUppercase{On some Fraïssé limits with Free Amalgamation}

\vspace{5mm}

    \large Yvon \textsc{Bossut}
\end{center}
\vspace{10mm}

Abstract : In this work a general way is given to construct some examples of NSOP$_1$ theories as limits of some Fraïssé class satisfying strong conditions. These limits will satisfy existence, that Kim-independence coincides with algebraic independence, and that forking independence is obtained by forcing base monotonicity on Kim-forking. These theories also come with a stationary independence relation. This study is based on results of Baudisch, Ramsey, Chernikov and Kruckman.\footnote{Partially supported by ANR-DFG AAPG2019 GeoMod}

\section{{\Large Introduction}}

NSOP$_1$ theories have recently been studied as a generalisation of simple theories. The notion of Kim-forking plays a similar role in NSOP$_1$ theories than forking in simple theories. Kim and Pillay \cite{kim1997simple} have shown that forking independence in simple theories is characterised by some of its properties (see \cite{wagner2000simple} for example). Chernikov and Ramsey have shown a similar result for Kim-independence in NSOP$_1$ theories, \cite[Theorem 5.8]{chernikov2016model}. The properties of Kim-forking in NSOP$_1$ theories have recently been studied by Ramsey, Kim, Chernikov, Kaplan (\cite{kaplan2020kim}, \cite{kaplan2019local} among others). In this work we shall study independence relations in NSOP$_1$ theories.

\vspace{10pt}
At the origin of this work is an interesting construction which appears in \cite{baudisch2002generic} : The parameterization of a theory. Given some theory $T$ in a language $L$, consider the language $L_{P}$ where we add a new sort $P$ and one variable in $P$ to every symbol of $L$, and denote by $O$ (for object) the original sort. Then an $L_{P}$-structure $A$ consists of its set of parameters $P(A)$ and of its set of objects $O(A)$ which has an $L$-structure $A_{p}:= O(A)_{p}$ induced by every parameter $p\in P(A)$. There does not always exist a model companion to the theory $T_{P}$ expressing that for every parameter $p$, the $L$-structure induced by $p$ is a model of $T$.

\vspace{10pt}
Examples of such model companions are given in \cite[Section 6.3]{chernikov2016model} and constructed as Fraïssé limits. Ramsey and Chernikov consider a relational language $L$ and a class $K$ of structures with the strong amalgamation property. Then the parameterized class $K_{P}$ is defined as the class of $L_{P}$-structures $A$ such that for every $p\in P(A)$ the $L$-structure induced by $p$ on $A$ is in $K$, and show that the limit of this class is NSOP$_1$. They give some condition on $K$ under which the limit of $K_{P}$ is non-simple. One example of the theories they build is the theory of parameterized equivalence relations. 

\vspace{10pt}
Here we shall give a context in which we can parameterize in the same fashion without assuming that the language is relational. The cost of this is to have strong assumptions about the generated structure, namely a set of condition $(H)$ which strengthens the usual Fraïsse properties. We obtain similar results as Ramsey and Chernikov : If some class $K$ satisfies $(H)$ then its limit $T$ is NSOP$_1$ and we can compute forking independence and Kim-independence in $T$ (\cref{nsop1ouais} and \cref{forkindep}).

\vspace{10pt}
This is done through the use of some stationary independence relation which comes from the assumption that $K$ is closed under free amalgamation in the categorical sense. This independence relation was already mentioned in \cite{baudisch2014free}, and it does not coincide with the notion of Free Amalgamation of \cite{conant2017axiomatic}, which also is a stationary independence relation but requires the algebraic closure to be generically trivial : If $A$ and $B$ are independent and algebraically closed then their union is algebraically closed - this will not hold in our case unless the algebraic closure is trivial. 

\vspace{10pt}
It is also shown that the condition $(H)$ is preserved under some modifications of the language $L$ and the class $K$ : parameterizing $K$ and adding generic structure to $K$. We also give a condition under which the limit of the parameterized class is non-simple. This context allows us to build model companions of some exotic theories, for example the theory of an abelian group with a unary function $f$ such that $f(0)=0$. We conclude by asking some questions about NSOP$_{1}$ theories with stationnary independence relations.

\vspace{10pt}
\textbf{Acknowledgements :} I thank Mark Kamsma and Nadav Meir for their useful questions and remarks.

\section{Preliminaries : Independence relations and NSOP1}

\subsection{Forking and dividing}

Let $\mathbb{M}\models T$ be a monster model of some complete theory. We introduce the different relations of independence that we will work with. See \cite{adler2005explanation} and \cite{d2023axiomatic} on the topic of the axiomatic approach to independence relations. 

\begin{definition} Let $b_{0},e$ be tuples of $\mathbb{M}$ and $\varphi(x,b_{0})$ be a formula.\begin{enumerate}
\item $\varphi(x,b_{0})$ \emph{divides over $e$} if there is an $e$-indiscernible sequence $I=(b_{i}$ : $i < \omega)$ such that $ \lbrace \varphi (x,b_{i})$ : $i<\omega \rbrace$ is inconsistent. In that case we say that $\varphi(x,b_{0})$ divides over $e$ with respect to $I$.
\item A partial type $p (x,b)$ \emph{forks over $e$} if it implies a finite disjunction of formulas each of which divides over $e$.
\item Let $a \in \mathbb{M}$. We write $a\forkindep^{f}_{e}b$ to denote the assertion that $\tp(a/eb_{0})$ does not fork over $e$ and $a\forkindep^{d}_{e}b_{0}$ to denote the assertion that $\tp(a/eb_{0})$ does not divide over $e$.
\item A sequence of tuples $(a_{i}$ : $i<\omega)$ is called \emph{$e$-Morley} if it is $e$-indiscernible and if $a_{i}\forkindep^{f}_{e}a_{<i} $ for every $i<\omega $.
\end{enumerate}
\end{definition}

\begin{definition} A set $e\subseteq \mathbb{M}$ is an \emph{extension basis} if $a\forkindep^{f}_{e}e$ for every $a\in \mathbb{M}$. A theory $T$ has \emph{existence} if every set is an extension basis. This is equivalent to saying that for every $e\in \mathbb{M}$ and every $p(x)\in S(e)$ there is an $e$-Morley sequence in $p$.
\end{definition}

\begin{remark} Every model is an extension basis, simple theories satisfy existence, and it is a conjecture that NSOP$_1$ theories do as well.
 \end{remark}

\begin{definition} Assume that $e$ is an extension basis. Let $b_{0}\in \mathbb{M}$ and $\varphi(x,b_{0})$ be a partial type.\begin{enumerate}

\item $\varphi(x, b_{0})$ \emph{Kim-divides over $e$} if there is an $e$-Morley sequence $(b_{i}$ : $i<\omega )$ such that $\lbrace \varphi(x,b_{i})$ : $i<\omega \rbrace$ is inconsistent. We will say that $\varphi(x, b_{0})$ Kim-divides over $e$ with respect to $(b_{i}$ : $i<\omega)$.

\item $p(x,b)$ \emph{Kim-forks over $e$} if it implies a finite disjunction of formulas each of which Kim-divides over $e$.

\item Let $a \in \mathbb{M}$. We write $a\forkindep^{K}_{e}b_{0}$ to denote the assertion that $\tp(a/eb)$ does not Kim-fork over $e$ and $a\forkindep^{Kd}_{e}b$ to denote the assertion that $\tp(a/eb)$ does not Kim-divide over $e$.

\item $I=(a_{i}$ : $i<\omega)$ is \emph{Kim-Morley over $e$} if it is $e$-indiscernible and if $a_{i}\forkindep^{K}_{e}a_{<i}$ for every $i<\omega$.
\end{enumerate}
\end{definition}

\begin{definition} We define \emph{algebraic independence} as the relation $A\forkindep^{a}_{C}B$ iff $  acl(AC)\cap acl(BC)=acl(C)$.\end{definition}

\begin{definition} We define \emph{$M$-independence} as the weakest relation implying $\forkindep^{a}$ and satisfying base monotonicity and closure, i.\ e.\  $A\forkindep^{M}_{C}B$ iff $acl(AC')\forkindep^{a}_{C'} acl(BC)$ for every $C\subseteq C' \subseteq acl(BC)$.\end{definition}

This relation was defined by Onshuus in the context of rosy theories (see \cite{onshuus2003thforking} for reference), but we will not make further use of rosiness because of \cref{thornfork}.

\subsection{NSOP1 and the Kim-Pillay theorem}

As many of the dividing lines the property of NSOP$_1$ is defined by using a tree.

\begin{definition} A formula $\varphi (x,y) $ in $\mathcal{L}$ has the strong order property of the first kind (SOP$_{1}$) in $T$ if there is a tree of parameters $(a_{\eta})_{\eta \in 2^{<\omega}}$ in a model of $T$ such that: \begin{center}
$\left\{ \begin{array}{l}
        \ \forall \eta \in 2^{\omega}$, $\lbrace \varphi (x,a_{\eta \lceil\alpha})$ : $\alpha < \omega \rbrace $ is consistent$  \\
        \ \forall \nu,\eta \in 2^{<\omega}$, if $\nu \unrhd \eta\frown 0  $ : $\lbrace \varphi (x,a_{\nu})$,  $\varphi (x,a_{\eta \frown 1}) \rbrace $ is inconsistent$
    \end{array}
    \right.  $ 
\end{center}
The theory $T$ has SOP$_{1}$ if one of its formulas does. $T$ is NSOP$_{1}$ otherwise.\end{definition}

\begin{figure}[!ht]
\centering
\resizebox{0.5\textwidth}{!}{%
\begin{circuitikz}
\tikzstyle{every node}=[font=\Huge]
\node [font=\Huge] at (12.75,-2.25) {$a_{\emptyset}
$};
\node [font=\Huge] at (7.5,2.75) {$a_{0}$};
\node [font=\Huge] at (17.5,2.75) {$a_{1}
$};
\node [font=\Huge] at (5.25,7.75) {$a_{0,0}$};
\node [font=\Huge] at (10,7.75) {$a_{0,1}$};
\node [font=\Huge] at (8.75,15.25) {....};
\node [font=\Huge] at (15,7.75) {$a_{1,0}$};
\node [font=\Huge] at (20.25,7.75) {$a_{1,1}$};
\draw [->, >=Stealth] (12.25,-1.75) -- (8,2.25);
\draw [->, >=Stealth] (13,-1.75) -- (17,2.25);
\draw [->, >=Stealth] (7.25,3.25) -- (5.25,7.25);
\draw [->, >=Stealth] (7.75,3.25) -- (9.75,7.25);
\draw [->, >=Stealth] (20,9.75) -- (20,12.25);
\draw [->, >=Stealth] (17.25,3.25) -- (15.5,7.25);
\draw [->, >=Stealth] (17.75,3.25) -- (19.75,7.25);
\node [font=\Huge, color={rgb,255:red,150; green,0; blue,0}] at (1.75,4.25) {Inconsistency between $a_{0,1}$};
\node [font=\Huge] at (20,12.75) {$a_{1,1\frown \eta}$};
\draw [line width=0.9pt, dashed] (20,8.25) -- (20,9.75);
\draw [ color={rgb,255:red,0; green,150; blue,0}, line width=0.9pt, short] (18.75,14) -- (21.25,14);
\draw [ color={rgb,255:red,0; green,150; blue,0}, line width=0.9pt, short] (21.25,14) -- (21.25,7.25);
\draw [ color={rgb,255:red,0; green,150; blue,0}, line width=0.9pt, short] (18.75,14) -- (18.75,7.75);
\draw [ color={rgb,255:red,0; green,150; blue,0}, line width=0.9pt, short] (18.75,7.75) -- (16.25,3.25);
\draw [ color={rgb,255:red,0; green,150; blue,0}, line width=0.9pt, short] (21.25,7.25) -- (18.25,2);
\draw [ color={rgb,255:red,0; green,150; blue,0}, line width=0.9pt, short] (16.25,3.25) -- (11.25,-2);
\draw [ color={rgb,255:red,0; green,150; blue,0}, line width=0.9pt, short] (18.25,2) -- (13,-3.5);
\draw [ color={rgb,255:red,0; green,150; blue,0}, line width=0.9pt, short] (13,-3.5) -- (11.25,-2);
\draw [ color={rgb,255:red,0; green,150; blue,0}, line width=0.9pt, ->, >=Stealth] (17,12.75) -- (18.25,12.75);
\node [font=\Huge, color={rgb,255:red,0; green,150; blue,0}] at (14.75,13.25) {Consistency};
\node [font=\Huge, color={rgb,255:red,0; green,150; blue,0}] at (14.75,12.25) {along a path
};
\draw [->, >=Stealth] (4.75,8.25) -- (2.75,12.25);
\draw [->, >=Stealth] (5.25,8.25) -- (7.25,12.25);
\node [font=\Huge] at (3,12.75) {$a_{0,0,0}$};
\node [font=\Huge] at (8,12.75) {$a_{0,0,1}$};
\node [font=\Huge] at (6.25,15.25) {$....$};
\node [font=\Huge] at (3.75,15.25) {$....$};
\node [font=\Huge] at (1.25,15.25) {$....$};
\draw [->, >=Stealth, dashed] (2.25,13.25) -- (1.5,14.75);
\draw [->, >=Stealth, dashed] (2.75,13.25) -- (3.75,14.75);
\draw [->, >=Stealth, dashed] (7,13.25) -- (6.25,14.75);
\draw [->, >=Stealth, dashed] (7.75,13.25) -- (8.5,14.75);
\draw [ color={rgb,255:red,150; green,0; blue,0} , line width=0.9pt ] (12,9) -- (9.25,9) -- (8,6.5) -- (10.75,6.5) -- cycle;
\draw [ color={rgb,255:red,150; green,0; blue,0}, line width=0.9pt, short] (3.75,6.5) -- (6.25,6.5);
\draw [ color={rgb,255:red,150; green,0; blue,0}, line width=0.9pt, short] (6.25,6.5) -- (8.75,11.5);
\draw [ color={rgb,255:red,150; green,0; blue,0}, line width=0.9pt, short] (3.75,6.5) -- (1.25,11.5);
\draw [ color={rgb,255:red,150; green,0; blue,0}, line width=1pt, dashed] (1.25,11.5) -- (0,14.25);
\draw [ color={rgb,255:red,150; green,0; blue,0}, line width=1pt, dashed] (8.75,11.5) -- (10,14.25);
\node [font=\Huge, color={rgb,255:red,150; green,0; blue,0}] at (1.5,3) {and any successor of $a_{0,0}$};
\end{circuitikz}
}%

\caption{The strong order property of the first kind}
\label{sop1}
\end{figure}

There are a lot of simplicity-like results characterizing NSOP$_1$-theories in terms of properties of Kim-forking. For example the following :

\begin{theorem}\label{symh} \cite[Corollary 4.9]{dobrowolski2022independence} A theory $T$ with existence is NSOP$_1$ if and only if Kim-forking satisfies symmetry in a monster model $\mathbb{M}$ of $T$, i.e. $a\forkindep^{K}_{e}b$ if and only if $b\forkindep^{K}_{e}a$ for any $a,b,e\in \mathbb{M}$.
\end{theorem}

However the result that we will use here to show that a theory is NSOP$_1$ is the Kim-Pillay theorem for Kim-forking over arbitrary sets :

\begin{theorem}\label{kimpillaykim} \cite[Theorem 5.1]{chernikov2023transitivitylownessranksnsop1} Assume that $T$ satisfies existence. The theory $T$ is NSOP$_1$ if and only
if there is an invariant independence relation $\forkindep$ on small subsets of the monster model which satisfies the following properties, for an arbitrary set of parameters
$e$ and arbitrary tuples from $\mathbb{M}$.

\begin{enumerate}
    \item Strong finite character: If $a\centernot\forkindep_{e}b$ then there is a formula $\varphi(x,b,e)\in \tp(a/eb)$ such that $a'\centernot\forkindep_{e}b$ whenever $\models \varphi(a',b,e)$.
    \item Existence and extension: For any $a,e,B$ there is $a'\equiv_{e}a$ such that $a'\forkindep_{e}B$.
    \item Monotonicity: $aa'\forkindep_{e}bb'$ implies $a\forkindep_{e}b$.
    \item Symmetry: $a\forkindep_{e}b$ implies $b\forkindep_{e}a$.
    \item The independence theorem: If $a_{0}\equiv^{L}_{e}a_{1}$, $a_{i}\forkindep_{e}b_{i}$ for $i=0,1$ and $b_{0}\forkindep_{e}b_{1}$ there is an $a\in \mathbb{M}$ such that $a\equiv_{eb_{i}}a_{i}$ for $i=0,1$ and $a\forkindep_{e}b_{0}b_{1}$.
\end{enumerate}
Then $T$ is NSOP$_1$ and $\forkindep$ strengthens $\forkindep^{K}$, i.e. if $a\forkindep_{e}b$ then $a\forkindep^{K}_{e}b$. If $\forkindep$ also satisfies:

\begin{enumerate}
    \item[6.] Transitivity: If $a\forkindep_{e}b$ and $a\forkindep_{eb}c$ then $a\forkindep_{e}bc$.
    \item[7.] Local character: If $\kappa \geq \vert T \vert^{+}$ is a regular cardinal, $(A_{i})_{i<\kappa}$ is an increasing continuous sequence of sets of size $< \kappa$, $A_{\kappa} = \bigcup_{i<\kappa}A_{i}$ and $\vert A_{\kappa} \vert = \kappa$, then for any finite $d$, there is some $\alpha < \kappa$ such that $d\forkindep_{A_{\alpha}}A_{\kappa}$.
\end{enumerate}
Then $\forkindep$ and $\forkindep^{K}$ coincide : $a\forkindep_{e}b$ if and only if $a\forkindep^{K}_{e}b$ for any $e,a,b\in \mathbb{M}$.
\end{theorem}

This result allows us to show that a theory is NSOP$_1$ and to characterize $\forkindep^{K}$ over arbitrary sets by using some independence relations that arise naturally from algebraic structures.

\section{Constructing the limit}

In this section we define a set of hypotheses $(H)$ on a class $K$ of structures. We begin by showing that if $K$ satisfies $(H)$ then we can construct the limit theory $T$ of $K$. We show that this limit $T$ is an NSOP$_1$ theory, has existence, that $\forkindep^{K}=\forkindep^{a}$, that $\forkindep^{M}=\forkindep^{f}=\forkindep^{d}$ over arbitrary sets and that it has weak elimination of imaginaries. We give some basic examples of such classes $K$. The bases of this study are \cite[Section 6.3]{chernikov2016model} and \cite{baudisch2002generic}. The work of Baudisch takes place in a much more general context, but we focus here on the study of independence relations, which appears in his paper in \cite[Theorem 4.1]{baudisch2002generic} under some stronger hypotheses.

\vspace{10pt}
In this section we will consider the notion of morphism of $L$-structures for a given language $L$. To avoid any confusion with the notion of $L$-embedding we will remind the definition of these two notions. Let $M$ and $N$ be two $L$-structures and let $\varphi : M \rightarrow N$ be a map. 

\begin{enumerate}
    \item $\varphi$ is a morphism of $L$-structures if for any atomic formula $\psi(x)$ in $L$ and any $x$-tuple $a\in M$, $M\models \psi(a)$ implies $N \models \psi( \varphi(a))$. 
    \item $\varphi$ is an $L$-embedding if for any atomic formula $\psi(x)$ in $L$ and any $x$-tuple $a\in M$, $M\models \psi(a)$ holds if and only if $N \models \psi( \varphi(a))$ does.
\end{enumerate}

An other way to say this is that a morphism of $L$-structures preserves the satisfaction of positive quantifier free formulas and that an $L$-embedding preserves the satisfaction of all quantifier free formulas.

\subsection{Hypothesis on the class and basic examples}

\begin{definition}\label{hypchap2} We consider a countable language $L$ and a class $K$ of finite and countable $L$-structures with the following properties, which we will refer to as $(H)$:
\begin{enumerate}
    \item \textbf{Universal Axiomatisability}: If an $L$-structure $A$ is not in $K$ there is a quantifier free formula $\varphi$ and a finite tuple $a\in A$ such that $A\models \varphi(a)$ and $A'\centernot\models \varphi(a')$ for every $a'\in A'\in K$.
    \item \textbf{Quantifier Elimination}: For every quantifier free formula $\varphi(x,y)$ satisfied by $(a,b)$ in some structure $A$ of $K$ there is a quantifier free formula $\psi (x)$ such that $A\models \psi(a)$ and if $a'\in A'\in K$ and $A'\models \psi (a')$ then there is $A'\subseteq B' \in K$ an extension and $b'\in B$ such that $B\models\varphi(a',b')$.
    \item \textbf{Free Amalgamation}: If $C\subseteq A,B$ are structures in $K$ (where we allow $C= \emptyset$ when the language $L$ does not have any constants) there is $D\in K$ and some embeddings $i_{A},i_{B}$ of $A$ and $B$ respectively to $D$ satisfying $i_{A}\vert_{C}=i_{B}\vert_{C}$, $i_{A}(A)\cap i_{B}(B)=i_{A}(C)$, $D$ is generated by the images of $A$ and $B$ and such that given any morphisms of $L$-structures $\varphi_{A}: A\rightarrow D'$, $\varphi_{B}: B\rightarrow D'$ to $D'\in K$ such that $\varphi_{A}\vert_{C}=\varphi_{B}\vert_{C}$ there is a unique morphism $\varphi: D\rightarrow D'$ such that $\varphi_{A}=\varphi\circ i_{A}$ and $\varphi_{B}=\varphi\circ i_{B}$. We will write $A\oplus_{C}B$ for this structure $D$. We also assume that the elements of $K$ agree on the structure generated by the constants, which is equivalent to the joint embedding property when Free Amalgamation holds.
    \item \textbf{Generic Element}: If $V$ is a sort there is a structure $\langle \lbrace x \rbrace \rangle\in K$ generated by a single element $x\in V$ such that for every structure $A\in K$ and element $a\in V(A)$ there is a morphism $\varphi: \langle\lbrace x \rbrace \rangle\rightarrow A \in K$ sending $x$ to $a$.
    \item \textbf{Algebraically Independent 3-Amalgamation}: Given a commutative diagram in $K$ as on the left below, where the arrows are embeddings, $\varphi_{B_{i}}(B_{i})\cap \varphi_{A_{i}}(A)= \varphi_{A_{i}}\circ i_{A}(E)= \varphi_{B_{i}}\circ i_{B_{i}}(E)$ for $i=0,1$ and $\psi_{B_{0}}(B_{0})\cap \psi_{B_{1}}(B_{1})= \psi_{B_{0}}\circ i_{B_{0}}(E)= \psi_{B_{1}}\circ i_{B_{1}}(E)$, we can complete it inside of $K$ to a commutative diagram as on the right (where the arrows are embeddings) such that $\varphi_{D_{i}}(D_{i})\cap \varphi_{B}(B)= \varphi_{D_{i}}\circ \varphi_{B_{i}}(B_{i})=\varphi_{B}\circ \psi_{B_{i}}(B_{i})$ for $i=0,1$ and $\varphi_{D_{0}}(D_{0})\cap \varphi_{D_{1}}(D_{1})= \varphi_{D_{0}}\circ \varphi_{A_{0}}(A)=\varphi_{D_{1}}\circ \varphi_{A_{1}}(A)$ (the blank arrows of the right diagram are those of the left diagram; their names have been omitted for clarity).
\end{enumerate}

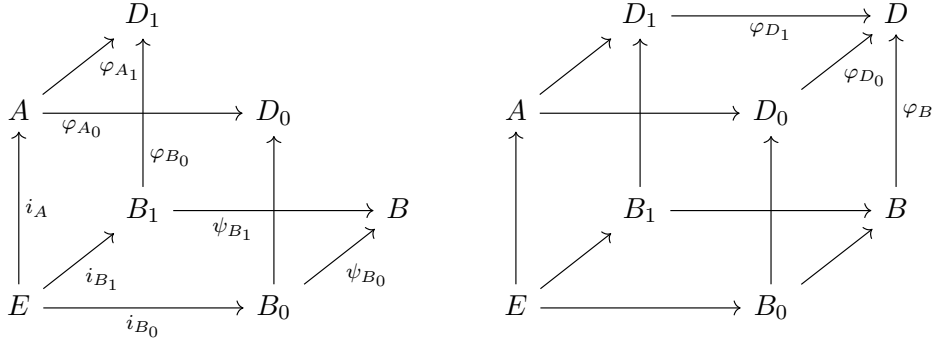
\begin{figure}[hbtp]
    \centering

\[\begin{tikzcd}
	& {D_{1}} &&&& {D_{1}} && D \\
	A && {D_{0}} && A && {D_{0}} \\
	& {B_{1}} && B && {B_{1}} && B \\
	E && {B_{0}} && E && {B_{0}}
	\arrow["{i_{B_{0}}}"', from=4-1, to=4-3]
	\arrow["{\psi_{B_{0}}}"', from=4-3, to=3-4]
	\arrow["{i_{B_{1}}}"', from=4-1, to=3-2]
	\arrow["{\varphi_{B_{0}}}"'{pos=0.2},from=4-3, to=2-3]
	\arrow["{i_{A}}"', from=4-1, to=2-1]
	\arrow["{\varphi_{A_{0}}}"'{pos=0.2}, from=2-1, to=2-3]
	\arrow["{\varphi_{B_{1}}}"'{pos=0.2}, from=3-2, to=1-2]
	\arrow["{\varphi_{A_{1}}}"'{pos=0.7}, from=2-1, to=1-2]
	\arrow["{\psi_{B_{1}}}"'{pos=0.3}, from=3-2, to=3-4]
	\arrow[from=4-5, to=4-7]
	\arrow[from=4-5, to=2-5]
	\arrow[from=2-5, to=2-7]
	\arrow[from=4-7, to=2-7]
	\arrow[from=4-7, to=3-8]
	\arrow[from=4-5, to=3-6]
	\arrow[from=2-5, to=1-6]
	\arrow["{\varphi_{D_{0}}}"', from=2-7, to=1-8]
	\arrow["{\varphi_{B}}"', from=3-8, to=1-8]
	\arrow[from=3-6, to=3-8]
	\arrow[from=3-6, to=1-6]
	\arrow["{\varphi_{D_{1}}}"', from=1-6, to=1-8]
\end{tikzcd}\]

    \caption{Algebraically Independent 3-Amalgamation}
\end{figure}										 
{\centering}
\end{definition}

Notice that we do not assume that there are countably many isomorphism classes (i.e. that the class $K$ is `essentially countable'), nor that the structures in $K$ are finitely generated. If $K$ satisfies Condition $1$ then $K$ has the hereditary property: If $A\subseteq B\in K$ is a substructure then $A\in K$. To check Condition $2\ $ for a formula $\varphi(x,y)$ realized by a tuple $(a,b)\in A$ it is enough to check it for a strengthening of this formula that is also realized by $(a,b)$ in $A$. 

\vspace{10pt}

Applying Condition $4\ $ to find $\langle \lbrace x \rbrace \rangle$ and then Condition $3$ to some $A\in K$ and $ \langle \lbrace x \rbrace \rangle$ above $\langle \emptyset \rangle$ gives us a structure $A\oplus_{\langle \emptyset \rangle} \langle \lbrace x \rbrace \rangle$ satisfying Condition $4\ $ above $A$, meaning that for every morphism $\varphi$ from $A$ to $B\in K$ and element $b\in V(B)$ there is a morphism $\varphi_{b} : A\oplus_{\langle \emptyset \rangle} \langle \lbrace x \rbrace \rangle \rightarrow B \in K$ sending $x$ to $b$ and extending $\varphi$.

\vspace{10pt}

We can define by induction $\langle \lbrace (x_{i})_{i<n}\rbrace \rangle$ for every $n<\omega$ the following way: $\langle \lbrace (x_{i})_{i<n+1}\rbrace \rangle$ is defined as $ \langle \lbrace (x_{i})_{i<n}\rbrace \rangle \oplus \langle \lbrace x_{n} \rbrace \rangle$. From this definition it is easy to show that $\langle \lbrace (x_{i})_{i<n}\rbrace \rangle$ satisfies the genericity property $4\ $ for $n$-tuples of elements of the sort $V$. Condition $1\ $ implies that $K$ is closed under unions of countable chains. With this and the previous remark we can define: $$\langle \lbrace (x_{i})_{i<\omega}\rbrace \rangle := \bigcup_{n<\omega}\langle \lbrace (x_{i})_{i<\omega}\rbrace \rangle.$$ 

This structure lies in $K$, since morphisms of $L$-structures are closed under unions of countable chains it is easy to show that it satisfies the genericity property $4\ $ for countable tuples of elements of the sort $V$.

\vspace{10pt}
The notation $ \langle \lbrace x \rbrace \rangle$ is here to insist on the fact that the set in between the braces is considered only as a set of generators. This notation $A\oplus_{\langle \emptyset \rangle} \langle\lbrace D \rbrace \rangle$ where $A,D\in K$ means the Generic Extension of $A$ with generators over $A$ indexed on $D$ and should not be confused with $A\oplus_{\langle \emptyset \rangle}  D $ which is the Free Amalgam of the structures $A,D\in K$ over the structure generated by the constants.

\vspace{10pt}

Condition $5\ $ is a strengthening of $3$-amalgamation for $\forkindep^{a}$; the fact that the images of $A$ and $B_{1}$ are disjoint over the image of $E$ in $D_{1}$ should be read as `$A$ and $B$ are algebraically independent inside of $D_{1}$' (which is written $A\forkindep^{a,D_{1}}_{E}B_{1}$ with the notations of \cite{aslanyan2023independence}, see this reference for some use of independence relation in categories). In the usual $3$-amalgamation we would ask that $D_{i}$ or $B$ are generated by the image of the two embeddings from $A$ and $B_{i}$ or from $B_{0}$ and $B_{1}$ respectively. Here however we allow ourselves to take an extension of the generated structure which will be useful when adding structure in Section 4.2.

\vspace{10pt}

 Some examples of classes $K$ satisfying $(H)$ are one based stable theories that can be presented as Fraïssé limits: The class of countable vector spaces over a finite (or countable) field (which is in the language), the class of countable abelian groups, the class of countable sets with no structure. In all those cases the conditions $1\ -4\ $ are easy to check, and the Condition $5\ $ is a consequence of the following proposition. An other example is the class of countable graph, which is not stable.

\begin{lemma}\label{feuvrier} Suppose that some class $K$ satisfies $1-4$ and that free independence and algebraic independence coincide, i.\ e.\  that for every $E\subseteq A,B\in K$ such that $A,B\subseteq D\in K$, the canonical map $\varphi : A\oplus_{E}B\rightarrow D$ is an embedding if and only if $A\cap B=E$. Then Condition $5\ $ holds in $K$. The classes associated to one based stable theories that we quoted previously satisfy this condition.
\end{lemma}

\begin{figure}[hbtp]
    \centering

\begin{tikzcd}
                                     &                             &                             &                                        & D_{1}\oplus_{B_{1}}B \arrow[rr] &                                 & D \\
                                     & D_{1} \arrow[rrru]          &                             & A\oplus_{E}B \arrow[ru] \arrow[rr]     &                                 & D_{0}\oplus_{B_{0}}B \arrow[ru] &   \\
A \arrow[ru] \arrow[rr] \arrow[rrru] &                             & D_{0} \arrow[rrru]          &                                        &                                 &                                 &   \\
                                     & B_{1} \arrow[uu] \arrow[rr] &                             & B \arrow[uu] \arrow[ruuu] \arrow[rruu] &                                 &                                 &   \\
E \arrow[uu] \arrow[rr] \arrow[ru]   &                             & B_{0} \arrow[ru] \arrow[uu] &                                        &                                 &                                 &  
\end{tikzcd}

    \caption{Proving Algebraically Independent 3-Amalgamation}
\end{figure}										 
{\centering}

\begin{proof} These hypotheses imply in particular that algebraic independence $\forkindep^{a,D}$ satisfies base monotonicity in the structures of $K$ (it is shown later in \cref{gammapropP} that free independence always satisfies base monotonicity, and it is clear that $\forkindep^{a,D}$ satisfies transitivity). Given $E,A,B_{0},B_{1},D_{0},D_{1}$ and a diagram satisfying the assumptions of $6$ consider the diagram above, where $D:=[D_{0}\oplus_{B_{0}}B]\oplus_{A\oplus_{E}B}[D_{1}\oplus_{B_{1}}B]$.

\vspace{10pt}
We want to show that $D_{0}\cap D_{1}=A$, $D_{0}\cap B=B_{0}$ and $D_{1}\cap B=B_{1}$ inside of $D$ (i.\ e.\ $D_{0}\forkindep^{a,D}_{A}D_{1}$, $D_{0}\forkindep^{a,D}_{B_{0}}B$ and $D_{1}\forkindep^{a,D}_{B_{1}}B$). We abuse notation and write $\langle AB\rangle$ for the substructure of $D$ generated by the images of $A$ and $B$, and for substructures $F,G,H$ of $D$ such that $F\cap G=H$, we will indifferently write $\langle FG\rangle$ and $F\oplus_{H}G$ for the generated structure.

\vspace{10pt}
The maps from $A\oplus_{E}B$ to $D_{i}\oplus_{B_{i}}B$ come from the fact that $A$ and $B$ are free over $E$ in $D_{i}\oplus_{B_{i}}B$. In fact if some element of $D_{i}\oplus_{B_{i}}B$ is in the image of both $A$ and $B$, then it is also in the image of $D_{i}$, so it is in the image of $B_{i}$, and by assumption $B_{i}$ and $A$ are free over $E$ in $D_{i}$. This diagram is commutative and all arrows are embeddings. If some element of $D$ is in the image of $D_{i}$ and $B$ then it is in the image of $B_{i}$, so $D_{i}\forkindep^{a,D}_{B_{i}}B$ for $i\in \lbrace 0,1\rbrace$. If some element of $D$ is in the image of $D_{0}$ and $D_{1}$ then it is in the image of $A\oplus_{E}B$ which is just the substructure generated by the images of $A$ and $B$.

\vspace{10pt}
Since $D_{i}\forkindep^{a,D}_{B_{i}}B$ by base monotonicity $D_{i}\forkindep^{a,D}_{\langle AB_{i}\rangle}\langle AB\rangle$, so it is enough to show that $\langle AB_{0}\rangle\forkindep^{a,D}_{A} \langle AB_{1}\rangle$. Now $A\forkindep^{a,D}_{E}B$ since $D_{0}\forkindep^{a,D}_{B_{0}}B$ and $A\forkindep^{a,D}_{E} B_{0}$. By base monotonicity $\langle AB_{0}\rangle\forkindep^{a,D}_{B_{0}} \langle B_{0}B_{1}\rangle$, since $B_{0}\forkindep^{a,D}_{E}B_{1}$ transitivity yields $\langle AB_{0}\rangle \forkindep^{a,D}_{E} B_{1}$, so $\langle AB_{0}\rangle \forkindep^{a,D}_{A} \langle AB_{1}\rangle$ by base monotonicity and Condition $5$ holds.\end{proof}

\subsection{Properties of the limit theory}

We begin by constructing the limit theory $T$ of $K$ and then we prove some of its properties. We shall assume the continuum hypothesis as we need an uncountable regular cardinal $\kappa$ such that $2^{<\kappa}=\kappa$, and it is convenient to take $\aleph_{1}$ as such.

\begin{definition} Given an $L$-structure $M$ and a cardinal $\kappa$ the \emph{$\kappa$-age of $M$} is the class of the substructures of $M$ of cardinality $<\kappa$.
\end{definition}

\begin{definition}
Let $\mathcal{K}$ be a class of countable structures in some countable language $\mathcal{L}$. An $\mathcal{L}$-structure $\mathcal{M}$ is \emph{$\mathcal{K}$-$\aleph_{1}$-saturated} if its $\aleph_{1}$-age is $\mathcal{K}$ and if for every $\mathcal{U},\mathcal{B}\in \mathcal{K}$ countable and any embeddings $f_{0} : \mathcal{U} \rightarrow \mathcal{M}$, $f_{1} : \mathcal{U} \rightarrow \mathcal{B}$ there is some embedding $g : \mathcal{B} \rightarrow \mathcal{M}$ such that $g\circ f_{1}=f_{0}$.
\end{definition}

\begin{prop}\label{construcMP}
There exists an $L$-structure $\mathcal{M}$ which is $K$-$\aleph_{1}$-saturated and has cardinal $\aleph_{1}$. Every structure of $\aleph_{1}$-age included in $K$ and of cardinality $\leq \aleph_{1}$ embeds into $\mathcal{M}$, and $\mathcal{M}$ is unique up to isomorphism.
\end{prop}

\begin{proof}
Let $(\mathcal{B}_{i})_{i<\aleph_{1}}$ enumerate the elements of $K$ up to isomorphism. We construct by induction an chain $(\mathcal{U}_{i})_{i<\aleph_{1}}$ of $L$-structures in $K$. Let $\mathcal{U}_{0}=\mathcal{B}_{0}$. If $\mathcal{U}_{i}$ has been defined, for every $k<\aleph_{1}$ let $(f^{\mathcal{U}_{i},k}_{j})_{j<\aleph_{1}}$ enumerate the embeddings of $\mathcal{U}_{i}$ into $\mathcal{B}_{k}$. We begin by showing that we can extend a countable number of partial embeddings in $K$.

\vspace{10pt}
\textbf{Claim:} If $\mathcal{U}\in K$, $B_{i}\in K$ for $i<\omega$ and $f_{i} : \mathcal{U}\rightarrow B_{i}$ are partial embeddings, then there is an extension $\mathcal{U}\subseteq \mathcal{U}'\in K$ and embeddings $g_{i} : B_{i}\rightarrow \mathcal{U}'$ such that $g_{i}\circ f_{i}$ is the identity on its domain for all $i<\omega$. 

\vspace{10pt}
\textit{Proof of the claim:} We build an increasing sequence $(\mathcal{U}'_{i})_{i<\omega}$ of extensions of $\mathcal{U}$. Let $E_{i}\subseteq \mathcal{U}$ be the domain of $f_{i}$ for $i<\omega$. Consider the extensions $E_{0}\subseteq \mathcal{U}$ and $f_{0} : E_{0} \rightarrow B_{0}$. By Free Amalgamation we can define $\mathcal{U}'_{0}=\mathcal{U}\oplus_{E_{0}}B_{0}\in K$, and we set $g_{0}$ to be the natural embedding $B_{0}\rightarrow \mathcal{U}'_{0}$. Assume that $\mathcal{U}'_{i-1}$ and $g_{i-1}$ are defined for $i-1<\omega$. Consider the extensions $E_{i}\subseteq \mathcal{U}'_{i-1}$ and $f_{i} : E_{i} \rightarrow B_{i}$. We can define $\mathcal{U}'_{i}=\mathcal{U}'_{i-1}\oplus_{E_{i}}B_{i}$, and we set $g_{i}$ to be the natural embedding $B_{i}\rightarrow \mathcal{U}'_{i}$. Now let $\mathcal{U}':=\bigcup\limits_{i<\omega}\mathcal{U}'_{i}$. Since $K$ is closed under unions of countable chains this is a structure in $K$ and it clearly satisfies the condition.\hspace*{0pt}\hfill\qedsymbol{}

\vspace{15pt}

We can now construct the chain. Assume that $\mathcal{U}_{j}$ has been defined for $j<i<\aleph_{1}$. Let $\mathcal{U}':= \bigcup\limits_{j<i}\mathcal{U}_{j}$. Consider the sequence of embeddings $(f^{\mathcal{U}_{j},k}_{l})_{j,k,l<i}$ seen as partial embeddings on $\mathcal{U'} \supseteq \mathcal{U}_{j}$. By the claim there is an extension $\mathcal{U}_{i} \supseteq \mathcal{U}'$ and a sequence of embeddings $(g^{\mathcal{U}_{j},k}_{l})_{j,k,l<i}: \mathcal{B}_{k} \rightarrow \mathcal{U}_{i}$ such that $g^{\mathcal{U}_{j},k}_{l}\circ f^{\mathcal{U}_{j},k}_{l}$ is the identity on its domain $\mathcal{U}_{j}$ for every $j,k,l<i$. We define $\mathcal{M}:= \bigcup\limits_{i<\aleph_{1}}\mathcal{U}_{i}$.

\vspace{10pt}

Now we show that $\mathcal{M} $ is $K$-$\aleph_{1}$-saturated. Let $\mathcal{U}$ be a countable substructure of $\mathcal{M}$. By isomorphic correction we can assume that the map $f_{0}$ of the saturation property is the inclusion. Then $\mathcal{U}\subseteq \mathcal{U}_{j}$ for some $j<\aleph_{1}$ so $\mathcal{U}\in K$ by Hereditarity thus $\aleph_{1}$-age($\mathcal{U}$)$\subseteq K$.

\vspace{10pt}
Consider an embedding $f: \mathcal{U} \rightarrow \mathcal{B} \in K$. $\mathcal{B}$ is isomorphic to $\mathcal{B}_{k}$ for some $k<\aleph_{1}$ via $\varphi : \mathcal{B} \rightarrow \mathcal{B}_{k}$. By Free Amalgamation there is some $k'<\aleph_{1}$ and some embeddings $g : \mathcal{B}_{k} \rightarrow \mathcal{B}_{k'}$ and $f': \mathcal{U}_{j} \rightarrow \mathcal{B}_{k'}$ such that $g \circ \varphi \circ f = f'$ (here $k'$ is such that $B_{k'}$ is isomorphic to $\mathcal{U}_{i}\oplus_{\mathcal{U}}\mathcal{B}_{k}$). Now $f': \mathcal{U}_{j} \rightarrow \mathcal{B}_{k'}$ is equal to $f^{\mathcal{U}_{j},k'}_{l}$ for some $l<\aleph_{1}$. Set $i=max\lbrace j,k,l\rbrace$, by construction, there is a $h=g^{\mathcal{U}_{j},k'}_{l} : \mathcal{B}_{k'} \rightarrow \mathcal{U}_{i+1}$ such that $h\circ f'$ is the inclusion $\mathcal{U}_{j}\subseteq\mathcal{U}_{i+1}$. Then $h\circ g \circ \varphi$ is the embedding we want.

\begin{figure}[hbtp]
    \centering
    
\begin{tikzcd}
\mathcal{U}_{i} \arrow[rrr, "\subseteq" description]                       &  &                                              & \mathcal{U}_{i+1}\subseteq \mathcal{M} \\
\mathcal{U}_{j} \arrow[u, "\subseteq" description] \arrow[rrr, "f'" description]             &  &                                              & \mathcal{B}_{k'} \arrow[u, "h"']       \\
\mathcal{U} \arrow[rr, "f" description] \arrow[u, "\subseteq" description] &  & \mathcal{B} \arrow[r, "\varphi" description] & \mathcal{B}_{k} \arrow[u, "g"']       
\end{tikzcd}
\caption{The embedding we needed}

\end{figure}

For the second point we proceed as in Fraïssé theory and only use the property of $\aleph_{1}$-saturation and the fact that we can amalgamate over the structure generated by the constants: if $\mathcal{M}'$ is a structure of age inside of $K$ and of cardinality $\leq \aleph_{1}$ we fix an enumeration $(m_{i})_{i<\aleph_{1}}$ of $\mathcal{M}'$ and define by induction an increasing sequence of embeddings $f_{i}:\langle m_{<i}\rangle \rightarrow \mathcal{M}$, then the union $f = \bigcup\limits_{i<\aleph_{1}} f_{i}:\mathcal{M} \rightarrow \mathcal{M}'$ is our embedding.

\vspace{10pt}
Let $f_{0}$ be the embedding of the constants from $\mathcal{M}'$ to $\mathcal{M}$ (possibly $f_{0}$ is the empty set). If $f_{j}$ is defined for $j<i<\aleph_{1}$, let $f':= \bigcup\limits_{j<i}f_{j} : \mathcal{B}'\subseteq \mathcal{M'}\rightarrow \mathcal{M}$.

\vspace{10pt}
Let $\mathcal{B}$ be the image of $M$ and $g' : \mathcal{B}\rightarrow \mathcal{B}'$ be the inverse map of $f'$, we can compose $g'$ with the inclusion $\mathcal{B}' \subseteq \langle m_{<i} \rangle$. Since $\langle m_{<i} \rangle \in K$ we can use the $\aleph_{1}$-saturation of $\mathcal{M}$ to find a map $f_{i}:\langle m_{<i}\rangle \rightarrow \mathcal{M}$ such that $f_{i}\circ g' =id_{B}$, which implies that $f_{i}\vert_{B'}=f$ by pre-composing with $f$, i.\ e.\  that $f_{i}$ extends $f_{j}$ for $j<i$. In a similar fashion, to show uniqueness we construct an isomorphism by back and forth.\end{proof}

Let $T$ be the theory of $\mathcal{M}$ constructed in \cref{construcMP}, we refer to $T$ as the \emph{limit theory} of $K$. It is a complete theory so we can consider a monster model $\mathbb{M}$ of $T$.

\begin{lemma}\label{qeparam} $T$ has quantifier elimination in $L$.
\end{lemma}

\begin{proof} This is clear from the Condition $2\ $: If we consider any $\aleph_{1}$-saturated model it is $K$-$\aleph_{1}$-saturated, so partial isomorphisms between finitely generated structures have the back and forth property.
\end{proof}

\begin{remark}\label{clotalg} A consequence of the Free Amalgamation property is that $dcl(A)=acl(A)=\langle A \rangle$ for every set $A\subseteq \mathbb{M}$. In fact if some element $a$ is not in the countable structure $A = \langle A \rangle$ then we can form $\langle A,a \rangle \oplus_{A} \langle A,a\rangle$ and get two copies of $a$ over $A$, repeating the process we get that $\qftp(a/A)$ is not algebraic.
\end{remark}

\begin{lemma}\label{amalgamalg} The relation of algebraic independence $\forkindep^{a}$ on subsets of $\mathbb{M}$ satisfies the independence theorem: If $E,A_{i},B_{i}\subseteq\mathbb{M}$ for $i\in \lbrace 0,1\rbrace$, $A_{0}\equiv_{E}A_{1}$, $A_{i}\forkindep^{a}_{E}B_{i}$ for every $i\in \lbrace 0,1\rbrace$ and $B_{0}\forkindep^{a}_{E}B_{1}$ there is an $A$ such that $A\equiv_{B_{i}}A_{i}$ for $i\in \lbrace 0,1\rbrace$ and $A\forkindep^{a}_{E}B_{0}B_{1}$.\end{lemma}

\begin{proof} Consider $E,A_{i},B_{i}$ for $i=0,1$ as in the statement, $\overline{x}_{A}$ a tuple of variables indexed by $A_{0}$ (and so also by $A_{1}$ through the isomorphism $\varphi: A_{0} \rightarrow A_{1}$ induced by $A_{0}\equiv_{E}A_{1}$). We write $x_{a}$ the variable in $\overline{x}_{A}$ associated to $a\in A_{0}$ and $\varphi(a)\in A_{1}$. We extend $\overline{x}_{A}$ into $\overline{x}_{A}\overline{y}_{0}$, a tuple of variable indexed on $\langle A_{0}B_{0}\rangle$ and into $\overline{x}_{A}\overline{y}_{1}$, a tuple of variable indexed on $\langle AB_{1}\rangle$ disjoint with $\overline{y}_{0}$.

\vspace{10pt}
Let $q_{i}(\overline{x}_{A}\overline{y}_{i}):=\qftp(\langle A_{i}B_{i}\rangle /B_{i})$ for $i=0,1$. We show that:

\begin{center}
$q_{0}(\overline{x}_{A}\overline{y}_{0})\cup q_{1}(\overline{x}_{A}\overline{y}_{1})\cup \lbrace t(\overline{x}_{\overline{a}})\centernot\in \langle B_{0}B_{1}\rangle$ : $t$ terms and $\overline{a}\in A_{0}$ tuples such that $t(\overline{a})\centernot\in E \rbrace$
\end{center}

is a consistent type. By compactness is it enough to consider finitely generated structures $E,A_{i},B_{i}$, and in that case we get consistency by just applying Condition $5$.
\end{proof}

\begin{definition} We define the \emph{$\Gamma$-independence} relation on small subsets of $\mathbb{M}$, the monster model of $T$. For $E\subseteq A,B$ say that $A\forkindep^{\Gamma}_{E}B$ if the structure $\langle AB\rangle$ and the inclusions $i_{A} :A  \rightarrow \langle AB\rangle $, $i_{B} : B \rightarrow \langle AB\rangle $ satisfy the universal property that every common extension of $A$ and $B$ in $\mathbb{M}$ factorizes uniquely through $\langle AB \rangle$. In the case where $E\subseteq A,B$ are countable $A\forkindep^{\Gamma}_{E}B$ if and only if $A\oplus_{E}B$ is isomorphic to $\langle AB \rangle$ via the canonical map of Condition $3$.
\end{definition}

\begin{lemma}\label{existenceamalg} Full existence: For every $E\subseteq A,B\subseteq \mathbb{M}$ there is $A'\equiv_{E}A$ such that $A'\forkindep^{\Gamma}_{E}B$.\end{lemma}

\begin{proof}
We want to construct an $L$-structure $D$ extending $B$ and an $E$-embedding $i_{A} : A\rightarrow D$ such that for $A':=i_{A}(A)$, $D=\langle A'B\rangle$ and $A'\forkindep^{\Gamma}_{E}B$. For countable substructures $E'\subseteq E$, $A'\subseteq A$, $B'\subseteq B$ we consider $A'\oplus_{E'}B'\in K$. The structure $D$ that we build is the inductive limit of these structures. We begin by showing that they indeed form an inductive system.

\begin{figure}[hbtp]
    \centering

\[\begin{tikzcd}
	&& {A'\oplus_{E'}B'} \\
	{A'} && {A''\oplus_{E''}B''} && {B'} \\
	{A''} && {E'} && {B''} \\
	&& {E''}
	\arrow[from=4-3, to=3-1]
	\arrow[from=4-3, to=3-5]
	\arrow[from=3-1, to=2-3]
	\arrow[from=3-5, to=2-3]
	\arrow[from=4-3, to=3-3]
	\arrow["{\exists!\varphi}"', dashed, from=2-3, to=1-3]
	\arrow[from=3-3, to=2-1]
	\arrow[from=3-1, to=2-1]
	\arrow[from=2-1, to=1-3]
	\arrow[from=3-5, to=2-5]
	\arrow[from=2-5, to=1-3]
	\arrow[from=3-3, to=2-5]
\end{tikzcd}\]

    \caption{The inductive system of $A'\oplus_{E'}B'$}

\end{figure}

Let $E''\subseteq E'\subseteq E$, $A''\subseteq A'\subseteq A$, $B''\subseteq B'\subseteq B$ be countable substructure. Then there is a canonical map from $A''\oplus_{E''}B''$ to $A'\oplus_{E'}B'$ where $E''\subseteq E'$, $A''\subseteq A'$ and $B''\subseteq B'$. We describe it in the above diagram. This map assures us that the system of the $A'\oplus_{E'}B'$ we are considering is directed and that we can define the $L$-structure $D$ as its direct limit. Also $\Diag(D/B)$ is consistent with $T$ because any formula that appears in it is satisfied in some of the $A'\oplus_{E'}B'\in K$.

\vspace{10pt}

Now we show that $D$ satisfies the right property: Let $D'\subseteq \mathbb{M}$, and $\varphi_{A} : A\rightarrow D'$, $\varphi_{B} : B \rightarrow D'$ form a commutative square over $E$. For every $E''\subseteq E'\subseteq E$, $A''\subseteq A'\subseteq A$, $B''\subseteq B'\subseteq B$ countable substructures the restricted morphism $\varphi_{E\vert E''}\subseteq \varphi_{E\vert E'}$, $\varphi_{A\vert A''}\subseteq \varphi_{A\vert A'}$, $\varphi_{B\vert B''}\subseteq \varphi_{B\vert B'}$ factories uniquely via $\varphi_{A''\oplus_{E''}B''}$ and $ \varphi_{A'\oplus_{E'}B'}$ respectively trough $A''\oplus_{E''}B''$ and $A'\oplus_{E'}B'$ respectively. By uniqueness if $\varphi$ is the canonical map from $A''\oplus_{E''}B''$ to $A'\oplus_{E'}B'$ then $ \varphi_{A'\oplus_{E'}B'} \circ\varphi = \varphi_{A''\oplus_{E''}B''}$, so the system of the $\varphi_{A'\oplus_{E'}B'}$ is also directed and we can consider its limit $\varphi_{A\oplus_{E}B}$, which is the map we want.\end{proof}

The difference with the previous free amalgam notion is that it was defined for the elements of $K$ only, which corresponds to the countable substructures of $\mathbb{M}$, whereas this notion is defined for small substructures. Also since our class $K$ is axiomatisable every countable substructure of $\mathbb{M}$ is inside of $K$ so there is no problem in considering the structures $A'\oplus_{E'}B'$ for any countable substructures $E',A',B'$ of $\mathbb{M}$.

\begin{prop}\label{gammapropP} The relation $A\forkindep^{\Gamma}_{E}B$ defined for structures $E\subseteq A,B$ is invariant, symmetric, satisfies existence, extension, monotonicity, finite character, stationarity, base monotonicity and transitivity.
\end{prop}

\begin{proof}
By existence here we mean that $A\forkindep^{\Gamma}_{E}E$ for every extension $E\subseteq A$. The properties of symmetry, existence and the fact that it implies $\forkindep^{a}$ are immediate. Stationarity is a consequence of quantifier elimination: Up to isomorphism over $B$ there is a unique structure $D=A\oplus_{E}B$ that satisfies $A\forkindep^{\Gamma}_{E}B$, which means a unique quantifier free type over $B$.

\vspace{10pt}
For monotonicity assume that $A\forkindep^{\Gamma}_{E}B$ and consider $E\subseteq C \subseteq B$ a substructure (this is enough by symmetry). We consider the free amalgam $D$ of $A$ and $C$ and the associated embeddings $i_{A},i_{C}$. Now let us consider the free amalgam $D'$ of $D$ and $B$ over $C$ and its associated embeddings $j_{D}$ and $j_{B}$. There is a unique morphism $\varphi_{D} : D \rightarrow \langle AB\rangle$ such that $\varphi_{D}\circ i_{A}=id_{A}$ and $\varphi_{D}\circ i_{C}=id_{C}$ and $\varphi_{D'}\circ j_{B} = id_{B}$. Since $\langle AB \rangle = A\oplus_{E}B$ there is a unique morphism $\psi_{D'} : D' \rightarrow \langle AB\rangle \rightarrow D'$ such that $\psi_{D'}$ restricted to $A$ and $B$ respectively is $j_{D}\circ i_{A}$ and $j_{B}$ respectively. From this we get that $\varphi_{D'}\circ \psi_{D'}=id_{D'}$ and $\psi_{D'}\circ \varphi_{D'} = id_{D}$, so $\varphi_{D'}$ is an isomorphism.

\vspace{10pt}
There is a unique morphism $\varphi_{D'} : D' \rightarrow \langle AB\rangle$ such that $\varphi_{D'}\circ j_{D} = \varphi_{D}$. So $\varphi_{D'}\circ j_{D}\circ i_{A}=id_{A}$, and we get that $D'$ is generated by the images of $A$ and $B$ and that $\varphi_{D'}$ is an isomorphism. We know that $D$ is generated by the images of $A$ and $C$, and that it is sent to $\langle AC \rangle \subseteq \langle AB\rangle$ by $j_{D}$, so $\varphi_{D'}\circ j_{D} : D \rightarrow \langle AC\rangle $ is an isomorphism by a similar argument, so $A\forkindep^{\Gamma}_{E}C$.

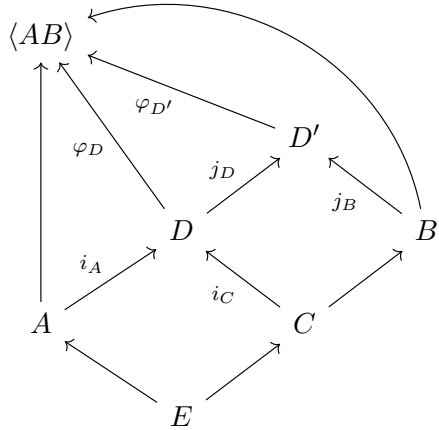
\begin{figure}[!ht]
\centering
\resizebox{0.8\textwidth}{!}{%
\begin{circuitikz}
\tikzstyle{every node}=[font=\large]

\draw [->, >=Stealth] (-1,8) -- (1,10)node[pos=0.5, fill=white]{$i_{C}$};
\node [font=\large] at (6.25,7.75) {A};
\node [font=\large] at (8.75,5.25) {E};
\node [font=\large] at (11.25,7.75) {C};
\node [font=\large] at (13.75,10.25) {C};
\node [font=\large] at (8.75,10.25) {$\langle AC\rangle$};
\node [font=\large] at (11.25,12.75) {$\langle AB \rangle$};
\node [font=\large] at (6.25,12.75) {D};
\draw [->, >=Stealth] (6.25,8) -- (6.25,12.5)node[pos=0.5, fill=white]{$\varphi_{A}$};
\draw [->, >=Stealth] (13.5,10.5) -- (11.5,12.5);
\draw [->, >=Stealth] (14,10.75) .. controls (14,14.5) and (10.5,15.5) .. (6.75,13.25) node[pos=0.5, fill=white]{$\varphi_{B}$};
\draw [->, >=Stealth] (11.5,8) -- (13.5,10);
\draw [->, >=Stealth] (11,8) -- (9,10);
\draw [->, >=Stealth] (8.5,5.5) -- (6.5,7.5);
\draw [->, >=Stealth] (9,5.5) -- (11,7.5);
\draw [->, >=Stealth] (9,10.5) -- (11,12.5);
\draw [->, >=Stealth] (6.5,8) -- (8.5,10);

\node [font=\large] at (3.75,7.75) {A};
\node [font=\large] at (1.25,5.25) {E};
\node [font=\large] at (-1.25,7.75) {C};
\node [font=\large] at (-3.75,10.25) {B};
\node [font=\large] at (1.25,10.25) {D};
\node [font=\large] at (-1.25,12.75) {D'};
\node [font=\large] at (3.75,12.75) {$\langle A B\rangle$};
\draw [->, >=Stealth] (3.5,8) -- (1.5,10)node[pos=0.5, fill=white]{$i_{A}$};
\draw [->, >=Stealth] (3.75,8) -- (3.75,12.5);
\draw [->, >=Stealth] (-1,12.75) -- (3,12.75)node[pos=0.5, fill=white]{$\varphi_{D'}$};
\draw [->, >=Stealth] (1,10.5) -- (-1,12.5)node[pos=0.5, fill=white]{$j_{D}$};
\draw [->, >=Stealth] (-4,10.75) .. controls (-4,14.5) and (-0.75,15.5) .. (3.25,13.25) ;
\draw [->, >=Stealth] (-1.5,8) -- (-3.5,10);
\draw [->, >=Stealth] (1,5.5) -- (-1,7.5);
\draw [->, >=Stealth] (1.5,5.5) -- (3.5,7.5);
\draw [->, >=Stealth] (1.5,10.5) -- (3.5,12.5)node[pos=0.5, fill=white]{$\varphi_{D}$};
\draw [->, >=Stealth] (-3.5,10.5) -- (-1.5,12.5)node[pos=0.5, fill=white]{$j_{B}$};
\draw [->, >=Stealth, dashed] (10.5,12.75) -- (6.5,12.75)node[pos=0.5, fill=white]{$\exists!\varphi$};
\draw [->, >=Stealth, dashed] (8.5,10.5) -- (6.5,12.5)node[pos=0.5, fill=white]{$\exists!\varphi_{\langle AC\rangle}$};
\end{circuitikz}
}%

\caption{Monotonicity (on the left) and transitivity (on the right) of $\forkindep^{\Gamma}$}
\label{lesdeuxfigures}
\end{figure}

For transitivity, given structures $E\subseteq A,C$, $C\subseteq B$ such that $A\forkindep^{\Gamma}_{E}C$ and $\langle AC \rangle \forkindep^{\Gamma}_{C}B$ we want to show that $A\forkindep^{\Gamma}_{E}B$.

Consider a structure $D$ and two embeddings $\varphi_{A},\varphi_{B}$ that coincide in $E$ from $A$ and $B$ respectively to $D$. Then the embeddings $\varphi_{A}$ and $\varphi_{B}\vert_{C}$ coincide over $E$, and by assumption there is a unique map $\varphi_{\langle AC\rangle} : \langle AC \rangle \rightarrow D$ such that $\varphi_{B}\vert_{C}= \varphi_{\langle AC\rangle}\vert_{C}$ and $\varphi_{A}=\varphi_{\langle AC\rangle}\vert_{A}$.

\vspace{10pt}
Now we consider the two maps $\varphi_{\langle AC\rangle} $ and $\varphi_{B}$ which coincide over $C$. Since $\langle AC \rangle\forkindep_{C}^{\Gamma}B$ there is a unique map $\varphi : \langle \langle AC\rangle B\rangle = \langle AB\rangle \rightarrow D$ such that $\varphi\vert_{B}=\varphi_{B}$ and $\varphi\vert_{\langle AC\rangle}=\varphi_{\langle AC\rangle}$, so in particular $\varphi\vert_{A}=\varphi_{A}$. This proof consist just in saying that the square formed by two adjacent pushouts forms a pushout.

\vspace{10pt}
For extension and stationarity, given structures $E\subseteq A,C$, $C\subseteq B$ such that $A\forkindep^{\Gamma}_{E}C$ we can form the free amalgam of $\langle AB \rangle$ and $C$ over $B$ inside of $\mathbb{M}$ thanks to \cref{existenceamalg}. This allows us to find some $A'$ such that $\langle A'C \rangle \cong \langle AC \rangle$ (since the free amalgam defines a unique isomorphism type), so by Q.\ E.\ $A' \equiv_{B}A$ and $\langle A'B \rangle\forkindep^{\Gamma}_{B}C$, so $A'\forkindep^{\Gamma}_{E}C$ by transitivity, and this gives us a unique type over $C$ for the same reason.

\vspace{10pt}
Finite character is also a direct consequence of the universal property: assume that $A\centernot\forkindep^{K}_{E}B$, so there is a structure $D$ and two embeddings $\varphi_{A}$, $\varphi_{B}$ from $A$ and $B$ respectively to $D$ that coincide in $E$ such that there is no map $\varphi : \langle AB\rangle \rightarrow D$ satisfying $\varphi_{A}=\varphi\vert_{A}$ and $\varphi_{B}=\varphi\vert_{B}$. This is equivalent to saying that the function that to an element $t(\overline{a},\overline{b})\in \langle AB\rangle$ associates $t(\varphi_{A}(\overline{a}),\varphi_{B}(\overline{b}))$ is not well defined (since this function is the only candidate) or that it is not a morphism. In either of those cases we can find some finite tuples $\overline{a}\in A$, $\overline{b}\in B$ such that the same thing is happening with the function that to an element $t(\overline{a}',\overline{b}')\in \langle E\overline{a}\overline{b}\rangle$ associates $t(\varphi_{A}(\overline{a}'),\varphi_{B}(\overline{b}'))$, so $\overline{a}\centernot\forkindep^{K}_{E}\overline{b}$.

\vspace{10pt}
For base monotonicity, given structures $E\subseteq A,C$ and $C\subseteq B$ such that $A\forkindep^{\Gamma}_{E}B$ we want to show that $\langle AC \rangle \forkindep^{\Gamma}_{C}B$. We take the same notations as for transitivity. As previously $A\forkindep^{\Gamma}_{E}C$ (by monotonicity this time). Consider a structure $D$ and two embeddings $\varphi_{\langle AC\rangle},\varphi_{B}$ from $\langle AC\rangle$ and $B$ respectively to $D$ that coincide on $C$. Then the embeddings $\varphi_{\langle AC\rangle}\vert_{A}$ and $\varphi_{B}$ coincide over $E$. By assumption there is a unique map $\varphi: \langle AB\rangle = \langle\langle AC \rangle B\rangle \rightarrow D$ such that $\varphi\vert_{A}= \varphi_{\langle AC\rangle}\vert_{A}$ and $\varphi\vert_{B}=\varphi_{B}$. Then we also have that $\varphi\vert_{C}=\varphi_{B}\vert_{C}=\varphi_{\langle AC\rangle}\vert_{C}$, so $\varphi\vert_{\langle AC\rangle}= \varphi_{\langle AC\rangle}$ since $A\forkindep^{\Gamma}_{E}C$. This proof consists in noticing that if the square formed by two adjacent squares is a pushout and the left one also is then the right one is also a pushout.\end{proof}

\begin{cor}\label{nsop1ouais} For every $E\subseteq A,B$ if $A\forkindep^{\Gamma}_{E}B$ then $A\forkindep^{f}_{E}B$, so $T$ has existence. Also $T$ is NSOP$_1$ and for every substructures $E\subseteq A,B$ we have $A\forkindep^{K}_{E}B$ if and only if $A\forkindep^{a}_{E}B$.\end{cor}

\begin{proof} We show that $A\forkindep^{\Gamma}_{E}B$ implies $A\forkindep^{d}_{E}B$. Assume that $A\forkindep^{\Gamma}_{E}B$ and let $(B_{i})_{i<\omega}$ be an $E$-indiscernible sequence with $B_{0}=B$. By extension we can find $A'\equiv_{B_{0}}A$ such that $A'\forkindep^{\Gamma}_{B_{0}}B_{<\omega}$. By transitivity and invariance $A'\forkindep^{\Gamma}_{E}B_{<\omega}$. By indiscernibility $B_{i}\equiv_{E}B_{0}$ for every $i<\omega$. Since $B_{0}\forkindep^{\Gamma}_{E}A'$ and $B_{i}\forkindep^{\Gamma}_{E}A'$ it follows from stationarity that $A'B_{i}\equiv_{E}A'B_{0}\equiv_{E}AB_{0}$ for every $i<\omega$. So $A\forkindep^{d}_{E}B$. Since $\forkindep^{\Gamma}$ satisfies extension we get that if $A\forkindep^{\Gamma}_{E}B$ then $A\forkindep^{f}_{E}B$, and existence for $\forkindep^{f}$ follows from existence for $\forkindep^{\Gamma}$.
\vspace{10pt}

Now that we know that $T$ has existence we can consider Kim-independence over arbitrary sets. We use the Kim-Pillay theorem for Kim-forking \cref{kimpillaykim}. By \cite[Lemma 2.7]{kruckman2018generic} $\forkindep^{a}$ satisfies extension, existence, monotonicity, symmetry, strong finite character, and witnessing. By \cref{amalgamalg} it satisfies the amalgamation theorem, and it is clear that it satisfies transitivity.

\vspace{10pt}
We now show that it satisfies the local character property ($7$ in \cref{kimpillaykim}). Let $\overline{a}$ be a finite tuple, $\kappa$ an uncountable regular cardinal, $(A_{i})_{i<\kappa}$ a continuous increasing sequence of structures of cardinality smaller than $\kappa$ and $A_{\kappa}=\bigcup\limits_{i<\kappa} A_{i}$. We want to show that there is an $i<\kappa$ such that $\langle A_{i}\overline{a}\rangle\forkindep^{a}_{A_{i}}A_{\kappa}$. We choose some $i_{0}<\kappa$. If $\langle A_{i_{0}}\overline{a}\rangle\centernot\forkindep^{a}_{A_{i_{0}}}A_{\kappa}$ the set $\langle A_{i_{0}}\overline{a}\rangle\cap A_{\kappa}$ is non empty and has cardinal smaller than $\kappa$, so it is contained in $A_{i_{1}}$ for some $i_{1}>i_{0}$. We iterate this process to define a sequence $(i_{j})_{j<\omega}$ and set $i=\bigcup\limits_{j<\omega}i_{j}$ which is smaller than $\kappa$ by regularity. Now if $a\in \langle A_{i}\overline{a}\rangle \cap A_{\kappa}$ it is in $\langle A_{i_{j}}\overline{a}\rangle \cap A_{\kappa}$ for some $j<\omega$ so by construction it is inside of $A_{i_{j+1}}\subseteq A_{i}$.
\end{proof}

\begin{remark} In the case where the limit $T$ of the class $K$ is a stable theory $\forkindep^{K}=\forkindep^{f}=\forkindep^{a}$ and $\forkindep^{a}=\forkindep^{\Gamma}$ since both of these relations are stationary and satisfy existence. Also the Kim-Pillay criterion for simple theories \cref{kimpillaykim} applied to the stationary relation $\forkindep^{\Gamma}$ tells us that $\forkindep^{\Gamma}$ satisfies local character if and only if $T$ is stable.
\end{remark}

\begin{remark} Notice that in order to get \cref{nsop1ouais} we only need $1-5$ and \cref{amalgamalg} to hold, and not the full generality of Algebraically Independent 3-Amalgamation. In our setting when applying the amalgamation theorem to $E,A_{0},A_{1},B_{0},B_{1}$ by \cref{amalgamalg} and Algebraically Independent 3-Amalgamation we have $AB_{i}\forkindep^{K}_{B_{i}}B_{0}B_{1}$ for $i=0,1$ and $AB_{0}\forkindep^{K}_{A}AB_{1}$. A proof is given in \cite[Proposition Proposition 2.2.1.6]{bossut2025looking} that this holds in any NSOP$_1$ theory with existence.
\end{remark}

\begin{remark} The proof of \cref{nsop1ouais} can be generalized to show that if $T'$ is a theory that admits a stationary independence relation $\forkindep$ satisfying the list of properties of \cref{gammapropP} and such that $\forkindep^{a}$ satisfies the independence theorem over arbitrary sets, then $\forkindep$ implies $\forkindep^{f}$, and so $T'$ is NSOP$_1$, has existence and $\forkindep^{a}=\forkindep^{K}$ over arbitrary sets.
\end{remark}

The following proposition is stated in \cite{kruckman2018generic} for the model companion of the empty theory in a language $L$ but it works more generally for any theory in which $\forkindep^{a}$ satisfies the independence theorem. This fact is part of some more general results about deducing weak elimination of imaginaries from the properties of some independence relations (see \cite[Section 4.4]{conant2019independence} for example).

\begin{prop}\cite[Proposition 3.22]{kruckman2018generic} If in a monster model of a complete theory $T'$ the relation $\forkindep^{a}$ satisfies the independence theorem then $T'$ has weak elimination of imaginaries.
\end{prop}

\begin{cor} $T$ has weak elimination of imaginaries.
\end{cor}

\begin{prop}\label{forkindep} For any $E\subseteq A,B$ we have $A\forkindep^{f}_{E}B$ iff $A\forkindep^{d}_{E}B$ iff $A\forkindep^{M}_{E}B$, where $A\forkindep^{M}_{E}B$ is defined as $\langle AC \rangle\forkindep^{a}_{C}B$ for every substructure $E\subseteq C \subseteq B$.
\end{prop}

\begin{proof} We begin by showing that in $\mathbb{M}$ the relation $\forkindep^{M}$ satisfies extension. Consider $E\subseteq A,B\subseteq \mathbb{M}$ such that $A\forkindep^{M}_{E}B$ and $B\subseteq B'$. We can find $A'$ such that $A'\equiv_{B}A$ and $\langle A'B\rangle \forkindep^{\Gamma}_{B}B'$, let us write $D'=\langle A'B'\rangle$. We show that $A'\forkindep^{M}_{E}B'$. Consider $C'\subseteq B'$ and define $C=C'\cap B$, we want to show that $\langle A'C'\rangle \cap B'=C'$. By Condition $5\ $, since $\langle A'C\rangle \cap C'=C$, $\langle A'C\rangle \cap B=C$ (because $A\forkindep^{M}_{E}B$) and $B\cap C'=C$, the following diagram with the solid arrows, where $\langle A'C'\rangle$ is the structure generated by $A'$ and $C'$ in $D'$, can be completed with the dashed arrows and a structure $D_{C'}$ such that $\varphi_{B'}(B')\cap \varphi_{\langle A'C'\rangle}(\langle A'C'\rangle)=\varphi_{B'}(C')=\varphi_{\langle A'C'\rangle}(C')$.

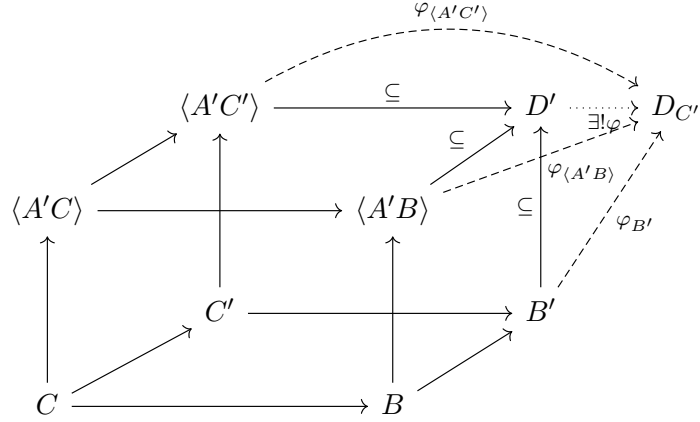
\begin{figure}[hbtp]
    \centering
    
\begin{tikzcd}
                                         & \langle A'C'\rangle \arrow[rr, "\subseteq"] \arrow[rrr, "\varphi_{\langle A'C'\rangle}", dashed, bend left] &                                                                                                 &  D' \arrow[r, "\exists!\varphi"', dotted] & D_{C'} \\
\langle A'C\rangle \arrow[rr] \arrow[ru] &                                                                                                             & \langle A'B\rangle \arrow[ru, "\subseteq"] \arrow[rru, "\varphi_{\langle A'B\rangle}"', dashed] &                                                                      &        \\
                                         & C' \arrow[uu] \arrow[rr]                                                                                    &                                                                                                 & B' \arrow[uu, "\subseteq"] \arrow[ruu, "\varphi_{B'}"', dashed]      &        \\
C \arrow[rr] \arrow[uu] \arrow[ru]       &                                                                                                             & B \arrow[ru] \arrow[uu]                                                                         &                                                                      &       
\end{tikzcd}

\caption{Characterizing $\forkindep^{f}$}

\end{figure}

There is a unique map $\varphi$ such that $\varphi\vert_{\langle A'B\rangle}=\varphi_{\langle A'B\rangle} $ and $\varphi\vert_{ B'}=\varphi_{ B'} $. The two maps $\varphi_{\langle A'C'\rangle}$ and $\varphi\vert_{\langle A'C'\rangle}$ coincide on the sets $C'$ and $\langle A'C\rangle$, and these sets generates $\langle A'C'\rangle$, so $\varphi\vert_{\langle A'C'\rangle}=\varphi_{\langle A'C'\rangle} $. If $t(a',c')=t'(b')$ for for two terms $t,t'$, $a'\in A'$, $c'\in C'$ and $b'\in B'$ inside of $D'$ then by applying $\varphi$ we get that $\varphi_{\langle A'C'\rangle}(t(a',c'))=\varphi_{B'}(t'(b'))$, so by Condition $5$ there is a $c\in C'$ such that $\varphi_{\langle A'C'\rangle} (c)=\varphi_{\langle A'C'\rangle}(t(a',c')) $ and $\varphi_{B'}(c)= \varphi_{B'}(t'(b'))$. Since the maps $\varphi_{B'}$ and $\varphi_{\langle A'C'\rangle}$ are embeddings we get that $t(a',c')=t'(b')\in B'$, so $\langle A'C'\rangle \cup B'=C'$. So $\forkindep^{M}$ satisfies extension.
\vspace{10pt}

It is always true that $\forkindep^{d}\implies \forkindep^{M}$. To show that $\forkindep^{M}\implies\forkindep^{f}$ it is sufficient to show that $\forkindep^{M}\implies \forkindep^{d}$ since $\forkindep^{M}$ satisfies extension. For this assume that $A\forkindep^{M}_{E}B_{0}$ and consider an indiscernible sequence $I=(B_{i})_{i<\omega}$. Let $p(x,B_{0}):= \tp(A/B_{0})$, we want to show that $\bigcup\limits_{i<\omega} p(x,B_{i})$ is consistent.

\vspace{10pt}
We begin defining some set $E' \supseteq E$ and an $E'$ indiscernible sequence $(B'_{i})_{i<\omega}$ such that $B_{i}\subseteq B'_{i}$ for all $i<\omega$, $B'_{i}\forkindep^{a}_{E'}B'_{<i}$ for every $i<\omega$, and some set $A' \supseteq A$. Let $q(x',B'_{0}):= \tp(A'/B'_{0})$ then it is sufficient to show that $\bigcup\limits_{i<\omega} q(x,B'_{i})$ is consistent.

\vspace{10pt}
Begin by adding a copy of $\omega$ at the beginning of the sequence $I:= (B_{i})_{i < \omega}$, i.e. we find $\Tilde{I}=(\Tilde{B}_{i})_{i<\omega}$ such that $\Tilde{I} \frown I$ is $E$-indiscernible. By extension we can assume that $A\forkindep^{M}_{E}\langle \Tilde{I}  \cup I\rangle$. Now set $E'=\langle \Tilde{I} \rangle$ to be the structure generated by $\Tilde{I}$. Consider the structures $B'_{i}:=\langle E'B_{i}\rangle$ for $i<\omega$. They form an $E'$-indiscernible sequence. Define $A' :=\langle E'A\rangle$, by base monotonicity $A'\forkindep^{M}_{E'}\langle (B'_{i})_{i<\omega}\rangle$. Now, if $b_{i}\in B'_{i}$ is also in $\langle B'_{<i}\rangle$ then it has to be in $E'$ by indiscernibility of $(\Tilde{B}_{i})_{i<\omega} \frown (B_{i})_{i < \omega}$ (this is the point of adding $(\Tilde{B}_{i})_{i<\omega}$: algebraic relations happening in the sequence are already happening inside of this part of the sequence). So we got the sequence and the sets we want.
\vspace{10pt}

Since $B'_{0}\equiv_{E'}B'_{i}$ there is some $\Tilde{A}_{i}$ such that $A'B'_{0}\equiv_{E'}\Tilde{A}_{i}B'_{i}$. We define by induction on $i<\omega$ a sequence of structures $A'_{i}$ such that $A'_{i}\forkindep^{a}_{E'}B'_{\leq i}$ and for all $j\leq i$ $A'B'_{0}\equiv_{E'}A'_{i}B'_{j}$. Let $A'_{0}=\Tilde{A}_{0}=A'$. If $A'_{i}$ has been defined then $A'_{i}\forkindep^{a}_{E'}B'_{\leq i}$, $\Tilde{A}_{i+1}\forkindep^{a}_{E'}B'_{i+1}$, $B'_{i+1}\forkindep^{a}_{E'}B'_{\leq i}$ and $\Tilde{A}_{i}\equiv_{E'} A'_{i}$, so by \cref{amalgamalg} we can find a suitable $A'_{i+1}$. By compactness there is some $A''$ such that $A''B'_{i}\equiv_{E'}A'B'_{0}$ for all $i<\omega$ and we get the consistency that we want.

\vspace{10pt}
To conclude if $A\forkindep^{M}_{E}B$ then $A\forkindep^{d}_{E}B$, by extension $A\forkindep^{f}_{E}B$ also holds, so these three notions are equivalent.\end{proof}

\begin{remark}\label{thornfork} We get from \cref{forkindep} that in $T$ thorn-forking coincides with forking, $T$ is rosy (which is defined as thorn-forking satisfying local character - see \cite{onshuus2003thforking} about rosy theories) if and only if $T$ is simple.
\end{remark}

In the previous proposition we insist on the fact that we are considering two algebraically closed sets $A,B$ containing the basis $E$; the reason for this is that in NSOP$_1$ theories it is not always true that $\forkindep^{f}=\forkindep^{d}$ over arbitrary sets. A counterexample is given in \cite[Section 3.2]{conant2023three}: if $T^{\centernot O}_{f}$ is the model completion of the empty theory in a language containing only a binary function and $T_{f}^{eq}$ its imaginary expansion, then in this paper a real tuple $a$, small model $M$ and some unordered pair $d$ such that $a\forkindep^{d}_{M}Md$ and $a\centernot\forkindep^{d}_{M}acl^{eq}(Md)$ are given.

\section{The parameterized class}

In this section we consider a Fraïssé class $K_{0}$ in a language $L_{0}$ and a second Fraïssé class $K_{P}$ in a language $L_{P}$ which is one sorted. We show that we can `parameterize' the class $K_{0}$ by $K_{P}$ to get a new class $K$ of structures in a new language that satisfies $(H)$ whenever $K_{0}$ and $K_{P}$ also does, and that in some of these cases the limit theory of $K$ is not simple, and so is strictly NSOP$_1$ by \cref{nsop1ouais}.

\subsection{The definition of the parameterized theory}

\begin{definition}
Let $L$ be a new language built from $L_{0}$ by adding a new sort $P$ equipped with an $L_{P}$-structure. We will refer to $P$ as the parameter sort, we will refer to the sorts of $L_{0}$ as the object sorts, denoted $O$. We replace every symbol of constant $c$ in $L_{0}$ of sort $O$ by a function symbol $\overline{c} : P \rightarrow O$, every function symbol $f : \prod_{i< n }O_{i}\rightarrow \prod_{i< m }O'_{i} $ by a function symbol $\overline{f} : P\times \prod_{i< n }O_{i}\rightarrow \prod_{i< m }O'_{i} $ and every relation symbol $R \subseteq \prod_{i< n }O_{i}$ by a relation symbol $\overline{R} \subseteq P\times \prod_{i< n }O_{i}$.
\end{definition}

Let us fix some notations: For a symbol of function $f\in L$ we will sometimes write $f_{p}(\overline{x})$ instead of $\overline{f}(p,\overline{x})$, and similarly for relation symbols. For an $L_{0}$-formula $\varphi$ and a parameter $p$ we will write $\varphi_{p}$ for the $L$-formula parameterized by $p$. For an $L$-structure $A$ and $p\in P(A)$ we write $A_{p}$ to denote the $L_{0}$-structure induced by $p$ on $O(A)$. For a set of objects $A$ and a tuple of parameters $\overline{p}$ we will write $\langle A \rangle_{\overline{p}}$ for the structure generated by $(A,\overline{p})$, we will also write $\langle A \rangle$ for the structure generated by a set $A$ (possibly containing parameters and objects), so $\langle A \rangle = \langle O(A)\rangle_{\langle P(A) \rangle}$.

\vspace{10pt}
We begin by citing some results of Baudisch on the parametrization of a theory showing that this construction works in a more general context than the one we are considering here. These results are proven essentially by decomposing quantifier-free $L_{P}$-formulas into a conjunction of $L$-formulas evaluated in different parameters. 

\begin{theorem} \cite[Theorem 2.9]{baudisch2002generic}: Let $T_{0}$ be a model complete theory which eliminates the quantifier $\exists^{\infty}$. Let $L_{P}$ be the empty language and $L$ be the parametrized language defined above. Let $T'$ be the $L$ theory stating that for every parameter $p\in P$ the $L_{0}$-structure induced by $p$ is a model of $T_{0}$. Then the theory $T'$ has a model companion $T$.
\end{theorem}

\begin{cor} \cite[Corollary 2.10]{baudisch2002generic} and \cite[Theorem 3.1]{baudisch2002generic} If $T_{0}$ is model complete and eliminates $\exists^{\infty}$ then $T$ is complete and eliminates $\exists^{\infty}$.
\end{cor}

\subsection{Basic properties}

After this small parenthesis we now return to the context of classes of $L$-structures and the set of assumptions $(H)$.

\begin{definition}
Let $K$ be the class of countable $L$-structures $A$ such that $A_{p}\in K_{0}$ for every $p\in P(A)$ and $P(A)\in K_{P}$.
\end{definition}

\begin{remark} We will most often consider the case where $L_{P}$ is the empty language and $K_{P}$ is the class of countable sets. It might be easier to read the proofs with the case of $L_{P}$ being the empty language in mind.
\end{remark}

\begin{lemma} \label{aximatizparam} If $K_{0}$ and $K_{P}$ satisfies Condition $1\ $ then $K$ also does.\end{lemma}

\begin{proof} If an $L$-structure $A$ is not in $K$ then either there is some parameter $p\in P(A)$ such that $A_{p}\centernot\in K_{0}$ or $P(A)\centernot\in K_{P}$. In the last case by $1$ in $K_{0}$ there is an $L_{0}$-formula $\varphi$ such that $A_{p}\models \varphi$ and such that no structures in $K_{0}$ satisfies $\varphi$. So $A\models \varphi_{p}$ and no structure in $K$ satisfies this formula. In the second case $P(A)\centernot\in K_{P}$ is expressed by an $L_{P}$ formula by Condition $1\ $ in $K_{P}$.\end{proof}

\begin{lemma}\label{freeamalgP} If $K_{0}$ and $K_{P}$ satisfies Conditions $1\ $, $3\ $ and $4\ $ then $K$ also does.
\end{lemma}

\begin{proof}
Condition $1\ $ follows from \cref{aximatizparam}. This condition is required here because we need $K_{0}$ to be closed under unions of countable chains. For Condition $3\ $ consider some structures $E\subseteq A,B\in K$. We construct an $L$-structure in $K$ satisfying the Free Amalgamation property. Let us write $E= (O(E),P(E))$, $A=(O_{A},P_{A})$ and $B=(O_{B},P_{B})$. We begin by forming the free amalgam $P(D):=P(A)\oplus_{P(E)}P(B)\in K_{P}$ which will be the parameter sort of the free amalgam in $K$.
\vspace{10pt}

Let $(p_{i})_{i < \omega}$ enumerate $P(D)$. We define a sequence $((D^{k}_{p_{i}})_{i<\omega})_{k <\omega} $ of tuples of countable $L_{0}$-structures indexed by $P(D)$ by induction such that the sequence $(D^{k}_{p})_{k<\omega}$ is a chain of $L_{0}$-structures for any $p\in P(D)$ and that for any $k<\omega$ the sequence $(D^{k}_{p_{i}})_{i<\omega}$ is an increasing sequence of sets. Let $D^{0}_{p_{-1}}:= A\dot\cup (B\setminus A)$ and let $E_{0}$ be the structure generated by the constants inside of $K_{0}$.
\vspace{10pt}

Using Conditions $3\ $ and $4\ $ in $K_{0}$ we define:\begin{align*}
D^{0}_{p_{i}}:= &(A_{p_{i}}\oplus_{E_{p_{i}}} B_{p_{i}})\oplus_{\langle \emptyset \rangle}\langle \lbrace D^{0}_{p_{i-1}}\setminus (D^{0}_{-1})\rbrace \rangle\ \text{if}\ p_{i}\in P(E),\\
D^{0}_{p_{i}}:= &A_{p_{i}}\oplus_{\langle \emptyset \rangle} \langle\lbrace D^{0}_{p_{i-1}}\setminus A\rbrace\rangle\ \text{if}\ p_{i}\in P(A)\setminus P(E),\\
D^{0}_{p_{i}}:= &B_{p_{i}}\oplus_{\langle \emptyset \rangle} \langle \lbrace D^{0}_{p_{i-1}}\setminus B\rbrace \rangle\ \text{if}\ p_{i}\in P(B)\setminus P(E),\\
D^{0}_{p_{i}}:= &E_{0}\oplus_{\langle \emptyset \rangle} \langle \lbrace D^{0}_{p_{i-1}}\rbrace\rangle \ \text{if} \ p_{i}\in P(D)\setminus (P(A)\cup P(B)).
\end{align*}

\vspace{10pt}

Set $D^{0}=\bigcup\limits_{i <\omega} D^{0}_{p_{i}}$. If $D^{k}$ is defined for some $1\leq k<\omega$ we set:

$$D^{k+1}_{p_{0}}= D^{k}_{p_{0}}\oplus_{\langle \emptyset \rangle}\lbrace  D^{k}\setminus D^{k}_{p_{0}} \rbrace$$.

Similarly if $D^{k}_{p_{i}}$ is defined for some $1\leq i <\omega$ we set:

$$D^{k+1}_{p_{i+1}}=D^{k}_{p_{i+1}}\oplus_{\langle \emptyset \rangle}\langle\lbrace  D^{k+1}_{i}\setminus D^{k}_{p_{i}} \rbrace \rangle.$$

Every $D^{k}_{p}$ is an $L_{0}$-structure for the parameter $p$. We set $O(D)=\bigcup\limits_{k<\omega} D^{k}$, this set is countable and has an $L_{0}$-structure for the parameter $p\in P(D)$ given by $D_{p}=\bigcup\limits_{k<\omega}D^{k}_{p}\in K_{0}$, so $D\in K$. This defines an $L$-structure $D=(O(D),P(D)) \in K$ which extends $A$ and $B$, and such that $A\cap B=E$ inside of $D$. Also if $A$ and $B$ are generated by $(\overline{a},P(A))$ and $(\overline{b},P(B))$ respectively $(\overline{ab},P(A),P(B))$ generates $D$.
\vspace{10pt}

Now that we constructed the structure $D$ we show that is satisfies the amalgamation property. Consider $D'\in K$, $\varphi_{A} : A \rightarrow D'$ and $\varphi_{B} : B \rightarrow D'$ as in the statement, and write $\varphi_{E}$ their restriction to $E$. The uniqueness of $\varphi$ is clear from the fact that $D$ is generated by the images of $A$ and $B$. For existence we define $\varphi$ on the parameters: By free amalgamation in $K_{P}$ there is a unique morphism $\varphi_{P}$ of $L_{P}$-structures from $P(D)$ to $P(D')$ such that the diagram commutes. We then construct an increasing sequence of tuples of morphism of $L_{0}$-structures indexed on $P(D)$: $$((\varphi^{k}_{p_{i}} : D^{k}_{p_{i}} \rightarrow D'_{\varphi(p_{i})})_{i<\omega})_{k<\omega}.$$

\vspace{10pt}
For $k=0$, if $p_{i}\in P(E)$ we take the unique morphism $\varphi^{0}_{p_{i}} : D^{0}_{p_{i}} \rightarrow D'$ that extends $\varphi_{A}(p,\mathunderscore) : A_{p} \rightarrow D'_{\varphi_{P}(p_{i})}$ and $\varphi_{B}(p,\mathunderscore) : B_{p} \rightarrow D'_{\varphi_{P}(p_{i})}$ and sends any element $d\in D^{0}_{p_{i-1}}\setminus D^{0}_{p_{-1}}$ to $\varphi^{0}_{p_{i-1}}(d)$. 

\vspace{10pt}
If $p_{i}\in P(A)\setminus P(E)$ we take the unique morphism $\varphi^{0}_{p_{i}} : D^{0}_{p_{i}} \rightarrow D'_{\varphi_{P}(p_{i})}$ that extends $\varphi_{A}(p,\mathunderscore) : A_{p} \rightarrow D'_{\varphi_{P}(p_{i})}$ and sends any element $d\in D^{0}_{p_{i-1}}\setminus A$ to $\varphi^{0}_{p_{i-1}}(d)$.

\vspace{10pt}
Similarly if $p_{i}\in P(B)\setminus P(E)$ we take the unique morphism $\varphi^{0}_{p_{i}} : D^{0}_{p_{i}} \rightarrow D'_{\varphi_{P}(p_{i})}$ that extends $\varphi_{B}(p,\mathunderscore) : B_{p_{i}} \rightarrow D'_{\varphi_{P}(p_{i})}$ and sends any element $d\in D^{0}_{p_{i-1}}\setminus B$ to $\varphi^{0}_{p_{i-1}}(d)$, and finally if $p_{i}\in P(D)\setminus(P(A)\cup P(B))$ we take the unique morphism $\varphi^{0}_{p_{i}} : D^{0}_{p_{i}} \rightarrow D'_{\varphi_{P}(p_{i})}$ that sends any element $d\in D^{0}_{p_{i-1}}$ to $\varphi^{0}_{p_{i-1}}(d)$.

\vspace{10pt}

Assume that for every $i<\omega$ the functions $\varphi^{k}_{p_{i}} : D^{k}_{p_{i}} \rightarrow D'_{\varphi(p_{i})}$ have been defined. Define $\varphi^{k+1}_{p_{0}}: D^{k+1}_{p_{0}} \rightarrow D'_{\varphi(p_{0})} $ as the unique morphism extending $\varphi^{k}_{p_{0}} : D^{k}_{p_{0}} \rightarrow D'_{\varphi(p_{0})} $ by sending $x\in D^{k}_{p_{i}}$ to $\varphi^{k}_{p_{i}}(x)$ for every $i<\omega$. We just repeat the same steps for the other parameters. Once all these functions are constructed we consider $\varphi = \bigcup\limits_{k<\omega}\varphi^{k}_{p}$ for every $p\in P(D)$. This function defines a map in $K$ from $D$ to $D'$ such that its restrictions to $A$ and $B$ are $\varphi_{A}$ and $\varphi_{B}$ respectively, and it is a morphism of $L$-structures since $\varphi = \bigcup \limits_{k<\omega} \varphi^{k}_{p}: D_{p} \rightarrow D'_{\varphi(p)}$ is a morphism of $L_{0}$-structures for every parameter $p\in P(D)$ and $\varphi_{P} : P(D)\rightarrow P(D')$ is a morphism of $L_{P}$-structures.
\vspace{10pt}

Checking Condition $4\ $ for some object sort in $K_{P}$ is immediate, since the structure formed by a single element in this sort will suffice. For the case of the parameter sort we consider the structure $\langle \lbrace p \rbrace \rangle \in K_{P}$ obtained by applying Condition $4\ $ in $K_{P}$. To the elements $q\in \langle \lbrace p \rbrace \rangle$ we associate disjoint copies of the $L_{0}$-structure generated by the constants and we amalgamate these copies freely, as in the previous proof, to get the $L$-structure $\langle \lbrace p \rbrace \rangle \in K$.\end{proof}

\begin{remark} In the proof of \cref{freeamalgP} we ahve shown that $K$ seen as a category has pushouts for pairs of monomorphisms. This proof also implies that $D$ does not depend (up to isomorphism) on the choice of the enumeration of the parameters: given the same construction for a different enumeration $D^{*}=\bigcup\limits_{i<\omega}D^{*,k}_{p_{i}}$, by the same proof $D^{*}$ also has the universal property so $D$ and $D^{*}$ are isomorphic over $A$ and $B$.
\end{remark}

\begin{remark} A substructure of a finitely generated structure is not necessarily finitely generated. If we consider the examples of parameterized infinite dimensional vector spaces over a finite field with $L_{P}$ being the empty language, the structure $E= \langle v, p_{1}p_{2}\rangle$ freely generated by two parameters and one object is infinite however the substructure $(O(E),p_{1})$ where we removed one parameter is not finitely generated.\end{remark}

\begin{lemma}\label{cond4} If $K_{0}$ and $K_{P}$ satisfies Conditions $1\ $, $2\ $, $3\ $ and $4\ $ then $K$ also does. For every quantifier-free formula $\varphi(\overline{x},\overline{a})$ satisfied in some structure of $K$ there is a quantifier-free formula $\psi (\overline{x})$ such that if $\overline{a}\in A\in K$ and $A\models \psi (\overline{a})$ then there is $A\subseteq B \in K$ an extension and $\overline{b}\in B$ such that $B\models\varphi(\overline{a},\overline{b})$.
\end{lemma}

\begin{proof} Conditions $1\ $, $3\ $ and $4\ $ follow from \cref{freeamalgP}, we need them here because we want to amalgamate structures in $K_{0}$ and as previously we need $K_{0}$ to be closed under countable unions. For Condition $2\ $ consider a quantifier free formula $\varphi(\overline{x}\overline{v},\overline{y}\overline{w})$ satisfied by some $\overline{a}\overline{p}\overline{b}\overline{q}\in A \in K$  with $\overline{x},\overline{y}$ in the object sorts and $\overline{v},\overline{w}$ in the parameters.

\vspace{10pt}
We can assume that $\varphi(\overline{x}\overline{v},\overline{y}\overline{w})$ is of the form:

$$\bigwedge_{i,j} E_{i}(f_{i},g_{i}) \wedge \bigwedge_{i} \overline{R}_{i}(v_{i},n_{i},m_{i}) \wedge \bigwedge_{i} \overline{R}'_{i}(w_{i},n'_{i},m'_{i}) \wedge \bigwedge_{i} \overline{R}'_{i}(w_{i}',n'_{i},m'_{i}),$$ where $R_{i},R_{i}'$ are relations symbol or negated relations symbols from $L_{0}$, $n_{i},n'_{i},m_{i},m'_{i}$ are terms in $\overline{x}\overline{v},\overline{y}\overline{w}$, $E_{i}$ are relations symbols or negated relation symbols form $L_{P}$ and $w_{i}',f_{i},g_{i}$ are terms in $\overline{v},\overline{w}$.

\vspace{10pt}

We consider a strengthening of $\varphi$ satisfied by $\overline{a}\overline{p}\overline{b}\overline{q}$. Let $(h_{i})_{i<m}$ enumerate the interpretation of all the sub-terms of the variable terms appearing in $\varphi$ when we fix $\overline{v}\overline{w}$ to be $\overline{p}\overline{q}$. Let $(c_{i})_{i<n}$ enumerate the interpretation of all the sub-terms of the object terms appearing in $\varphi$ when we fix $\overline{x}\overline{v}\overline{y}\overline{w}$ to be $\overline{a}\overline{p}\overline{b}\overline{q}$. Let $\overline{k}=(k_{i})_{i<m}$ and $\overline{z}=(z_{i})_{i<n}$ be tuples of variable of the same length. Enlarging $\overline{a}\overline{p}$ we assume that it contains every term $c_{i}$ and $h_{i}$ such that $c_{i}\in \langle \overline{a}\overline{p}\rangle$ and $h_{i}\in \langle\overline{p}\rangle$, let $\overline{x}\overline{v}$ be the corresponding tuples of variables.

\vspace{10pt}
We replace the tuple $\overline{b}\overline{q}$ by $\overline{b}\overline{c}\overline{q}\overline{h}$ and write $\varphi = \varphi(\overline{x}\overline{v},\overline{y}\overline{z}\overline{w}\overline{k})$. We add to $\varphi$ the formulas that expresses the fact that $k_{i}$ is a sub-term, meaning that if $h_{i}=t_{i}(\overline{p},\overline{q},h_{<i})$ for $t_{i}$ an $L_{P}$-term we add $k_{i}=t_{i}(\overline{x},\overline{y},k_{<i})$ to $\varphi$ and replace the occurrences of $t_{i}(\overline{x},\overline{y},k_{<i})$ by occurrences of $k_{i}$.

\vspace{10pt}
Similarly we add to $\varphi$ the formulas that expresses the fact that $z_{i}$ is a sub-term, meaning that if $c_{i}=\overline{t}_{i}(\rho,\overline{a},\overline{b},c_{<i})$ for $\overline{t}_{i}$ an $L_{0}$-term and $\rho\in \overline{p}\overline{b}\overline{h}$ we add $z_{i}=\overline{t}_{i}(\nu,\overline{x},\overline{y},z_{<i})$ to $\varphi$ where $\nu \in \overline{v}\overline{w}\overline{k}$ is the variable corresponding to $\rho$ and replace the occurrences of $\overline{t}_{i}(\nu,\overline{x},\overline{y},z_{<i})$ by occurrences of $z_{i}$. After this the object terms that appears in $\varphi$ only use one parameter.

\vspace{10pt}
For every parameter variable $\nu \in \overline{v}\overline{w}\overline{k}$ consider the $L_{0}$-formula $\varphi_{\nu}(\overline{x},\overline{y}\overline{z})$ that specifies every equality $z_{j}=\overline{t}_{j}(\nu,\overline{x},\overline{y},z_{<j})$ (we consider the previous equalities and only mention those where the parameter $\nu$ is used), the relations $R_{j,\nu}$ that hold between the elements of $\overline{a}\overline{b}\overline{c}$ (so here we look at the $L_{0}$-structure induced by $\nu$), and we include the instances of equality and inequality among these relations.

\vspace{10pt}
We consider the formula $\varphi_{P}$ that specifies every equality $k_{i}=t_{i}(\overline{x},\overline{y},k_{<i})$ (we consider the previous equalities and only mention those in the parameter sort), the relations $E_{j}$ that holds between the elements of $\overline{p}\overline{h}$, and we include the instances of equality and inequality among these relations. We can write the formula $\varphi$ as:

$$\varphi(\overline{x}\overline{v},\overline{y}\overline{z}\overline{w}\overline{k})= \varphi_{P}(\overline{v},\overline{w}\overline{k})\wedge \bigwedge\limits_{\nu\in \overline{v}\overline{w}\overline{k}}\varphi_{\nu}(\overline{x},\overline{y}\overline{z}).$$

We can consider $\psi_{P}$ which is the $L_{P}$-formula associated to $\varphi_{P}$ by $2$ in $K_{P}$ and $\psi_{\nu}$ the $L_{0}$-formula associated to $\varphi_{\nu}$ by Condition $2$ in $K_{0}$ for every $\nu \in \overline{v}\overline{w}\overline{k}$. Let $\psi(\overline{x},\overline{v}):= \psi_{P}(\overline{v}) \wedge \bigwedge_{\nu\in \overline{v}}\varphi_{\nu}(\overline{x})$. It satisfies the condition, and we conclude as in \cref{cond4}: If $\psi(\overline{x},\overline{v})$ is satisfied by $\overline{a}'\overline{p}'\in A'\in K$ we can extend $P(A)$ in $K_{P}$ thanks to $\psi_{P}$, choose an extension for every parameter variable $\overline{v}\overline{w}\overline{k}$, and amalgamate these extension in an $L$-structure extending $A$ as in \cref{freeamalgP}.\end{proof}

\subsection{Amalgamation and Combinatorial Properties}

In this subsection we assume that $K_{0}$ and $K_{P}$ both satisfy $(H)$. By the previous section we can construct the limit theory $T_{0}$ of $K_{0}$ and $\mathbb{M}_{0}$ a monster model of it. $T_{0}$ has quantifier elimination and Kim-independence and algebraic independence coincide inside of it, similarly we can construct the limit theory $T_{P}$ of $K_{P}$ and a monster model $\mathbb{M}_{P}$.

\vspace{10pt}
By \cref{clotalg} the algebraic closure and the generated structure coincide in $\mathbb{M}_{0}$ we can represent non-Kim-forking extensions as commutative squares as in the following figure:

\begin{figure}[hbtp]
    \centering
    
\[\begin{tikzcd}
	A && D \\
	\\
	E && B
	\arrow["i", from=3-1, to=3-3]
	\arrow["g"', from=3-3, to=1-3]
	\arrow["j"', from=3-1, to=1-1]
	\arrow["f", from=1-1, to=1-3]
\end{tikzcd}\]

\caption{A commutative square}

\end{figure}

The arrows are embeddings, the structure $D$ is generated by the images $f(A)$ and $g(B)$, and $f(A)\cap g(B)=f\circ j(E)=g\circ i(E)$. Then the isomorphism type of $f(A)$ over $g(B)$ gives us a type over $g(B)$ which corresponds to a non-Kim-forking extension of $\tp(A/E)$. We will refer to such squares as strong amalgams. However for the (Algebraically Independent 3-Amalgamation) the squares we are considering consist in some `extensions' of such squares, meaning that we compose with an embedding from $D$ to some $D'\in K_{0}$.

\begin{prop}\label{amalgamationaP}
    $K$ satisfies Algebraically Independent 3-Amalgamation.
\end{prop}

\begin{proof}
Consider the following diagram in $K$, which is as in the definition of Algebraically Independent 3-Amalgamation and where for simplicity we take the embeddings to be inclusions. We want to complete it.

\begin{figure}[hbtp]
    \centering

\[\begin{tikzcd}
	& {D_{1}} &&&& {D_{1,p_{j}}} && D^{*,0}_{p_{j}} \\
	A && {D_{0}} && A_{p_{j}} && {D_{0,p_{j}}} \\
	& {B_{1}} && B && {B_{1,p_{j}}} && B_{p_{j}} \\
	E && {B_{0}} && E_{p_{j}} && {B_{0,p_{j}}}
	\arrow["{i_{B_{0}}}"', from=4-1, to=4-3]
	\arrow["{\psi_{B_{0}}}"', from=4-3, to=3-4]
	\arrow["{i_{B_{1}}}"', from=4-1, to=3-2]
	\arrow[from=4-3, to=2-3]
	\arrow["{i_{A}}"', from=4-1, to=2-1]
	\arrow["{\varphi_{A_{0}}}"'{pos=0.2}, from=2-1, to=2-3]
	\arrow["{\varphi_{B_{0}}}"'{pos=0.2}, from=3-2, to=1-2]
	\arrow["{\varphi_{A_{1}}}"'{pos=0.7}, from=2-1, to=1-2]
	\arrow["{\psi_{B_{1}}}"'{pos=0.3}, from=3-2, to=3-4]
	\arrow[from=4-5, to=4-7]
	\arrow[from=4-5, to=2-5]
	\arrow[from=2-5, to=2-7]
	\arrow[from=4-7, to=2-7]
	\arrow[from=4-7, to=3-8]
	\arrow[from=4-5, to=3-6]
	\arrow[from=2-5, to=1-6]
	\arrow["{\varphi_{D_{0}}}"', from=2-7, to=1-8]
	\arrow["{\varphi_{B}}"', from=3-8, to=1-8]
	\arrow[from=3-6, to=3-8]
	\arrow[from=3-6, to=1-6]
	\arrow["{\varphi_{D_{1}}}"', from=1-6, to=1-8]
\end{tikzcd}\]

\caption{Proving Algebraically Independent 3-Amalgamation in $T_{P}$}

\end{figure}
\vspace{10pt}
Firstly by Condition $5\ $ in $K_{P}$ we can complete the parameter part of this diagram with an $L_{P}$-structure $P(D)\in K_{P}$ which is going to be the parameter sort of our structure $D$. After this we begin by choosing a completion for every parameter. Fix an enumeration $(p_{i})_{i<\omega}$ of $P(D)$. Looking at the $L_{0}$-structure indexed by any of the parameters $p_{j}\in P(E)$ we can find an $L_{0}$-structure $D^{*,0}_{p_{j}}$ as in the right diagram. We begin by filling the cube for the $L_{0}$-structures induced by each of the parameters.

\vspace{10pt}
These sets $D^{*,0}_{p_{j}}$ for $p_{j}\in P(E)$ are also equipped with the $L$-structure on the sets $\varphi_{B}(B)$ and $\varphi_{D_{i}}(D_{i})$ for $i=0,1$ induced by the embeddings. We construct the $L_{0}$-structure $D^{*,0}_{p_{j}}$ for $j\in P(D)\setminus P(E)$ and consider the natural mappings from $D_{0}$, $D_{1}$ and $B$ to $D^{*,0}_{p_{j}}$.\begin{align*}
&\text{If}\  p_{j}\in P(A)\setminus P(E)\ \text{set}\ D^{*,0}_{p_{j}}=  [ D_{0,p_{j}}\oplus_{A_{p_{j}}}D_{1,p_{j}}]\oplus_{\langle \emptyset \rangle}\langle \lbrace B \setminus (B_{0}\cup B_{1})\rbrace \rangle\\
&\text{If}\  p_{j}\in P(B_{i})\setminus P(E)\ \text{set}\ D^{*,0}_{p_{j}}=  [ D_{0}\oplus_{B_{0}}B]\oplus_{\langle \emptyset \rangle}\langle \lbrace D_{1} \setminus (A\cup B_{1})\rbrace\rangle\\
&\text{If}\  p_{j}\in P(D_{0})\setminus (P(A)\cup P(B_{0}))\ \text{set}\ D^{*,0}_{p_{j}}=  D_{0,p_{j}}\oplus_{\langle \emptyset \rangle}\langle \lbrace (B \setminus B_{0})\cup (D_{1}\setminus A)\rbrace\rangle\\
&\text{If}\  p_{j}\in P(B)\setminus (P(B_{0})\cup P(B_{1}))\ \text{set}\ D^{*,0}_{p_{j}}=  B_{p_{j}}\oplus_{\langle \emptyset \rangle}\langle\lbrace (D_{0}\setminus B_{0}) \cup (D_{1}\setminus B_{1})\rbrace\rangle\\
&\text{If}\  p_{j}\in P(D)\setminus (P(D_{0})\cup P(D_{1})\cup P(B))\ \text{set}\ D^{*,0}_{p_{j}}=  E_{0}\oplus_{\langle \emptyset \rangle}\langle\lbrace (D_{0} \cup D_{1}\cup B)\rbrace \rangle
\end{align*}

In this last union we assimilate elements that have the same predecessors in the diagram. Now that we have all of our $L_{0}$-structures we will just put them together in a canonical way: we define an increasing sequence of tuples of $L_{0}$-structures $((D^{k}_{p_{j}})_{j<\omega})_{k<\omega}$.

\vspace{10pt}
Let $D^{0}_{p_{0}}=D^{*,0}_{p_{0}}$, if $D^{0}_{p_{j}}$ has been defined for some $j<\omega$ set:

$$D^{0}_{p_{j+1}}:=D^{*,0}_{p_{j+1}}\oplus_{\langle \emptyset \rangle}\langle \lbrace D^{0}_{p_{j}} \setminus (D_{0} \cup D_{1} \cup B \rangle) \rbrace\rangle,$$

where we assimilate $D_{0}$, $ D_{1}$ and $B$ with their image in $D^{0}_{p_{j}}$. Once $D^{0}_{p_{j}}$ has been defined for every $j<\omega$ we set $D^{0}:= \bigcup\limits_{i<\omega}D^{0}_{p_{j}}$. If $D^{k}$ is defined we set: 

$$D^{k+1}_{p_{0}}:= D^{k}_{p_{0}}\oplus_{\langle \emptyset \rangle}\langle \lbrace D^{k} \setminus D^{k}_{p_{0}}\rbrace \rangle.$$

Now assume that $D^{k+1}_{p_{j}}$ is defined for some $k<\omega$,$j<\omega$, we set:

$$D^{k+1}_{p_{j+1}}:= D^{k}_{p_{j+1}}\oplus_{\langle \emptyset \rangle}\langle \lbrace D^{k+1}_{p_{j}} \setminus D^{k}_{p_{j+1}}\rbrace \rangle.$$

Then take $D:=\bigcup\limits_{k<\omega} D^{k}_{p_{j}}$ with the $L_{0}$-structure induced for every $j< \omega$. This defines an $L$-structure. The natural mappings from $D_{0}$, $D_{1}$ and $B$ to $D$ are embeddings for the $L$-structure we put on $D$ and they satisfy the conditions.\end{proof}

Finally we can state that the class $K$ satisfies the conditions of $(H)$ by \cref{amalgamationaP}, \cref{aximatizparam} and \cref{cond4}. We can define the limit theory $T$ of $K$ and take $\mathbb{M}$ a monster model of $T$. By \cref{nsop1ouais} we know that $T$ is an NSOP$_1$ theory with existence such that $\forkindep^{K}=\forkindep^{a}$ over arbitrary sets.

\begin{prop}\label{qecriterP}
    $T$ implies the following sets of formulas for any $n<\omega$:
    \begin{enumerate}
        \item If $\overline{a}$ is an $m$-tuple of objects, $\pi_{i}(\overline{x})$ is an $m$-type in $T_{0}$ for every $i\leq n$ such that $\pi_{i} \vdash x_{j}\centernot= x_{j'}$ for $j\centernot=j'$ and $q(\overline{v})$ in a type in $T_{P}$, then there exists an infinite number of tuples of parameters $(p_{i})_{i\leq n}$ such that $\models \pi_{i}(\overline{a}_{p_{i}})$ for every $i\leq n$ and $\models q((p_{i})_{i\leq n})$.

        \item If $\overline{p}=p_{0},..,p_{n}$ is a $n$-tuple of parameters and $\pi_{i}(\overline{x},\overline{b}_{p_{i}})$ is a type in $\mathcal{M}_{p_{i}}$ over a $n$-tuple of objects $\overline{b}$ such that the induced type in any of the $x_{i}$ is non-algebraic and that $\pi_{i} \vdash x_{j}\centernot= x_{j'}$ for $j\centernot=j'$ then there is a tuple of objects $\overline{a}$ such that $\models \pi_{i}(\overline{a}_{p_{i}},\overline{b}_{p_{i}})$ for every $i\leq n$.
    \end{enumerate}
\end{prop}

\begin{proof} Consider the limit $\mathcal{M}$ of $K$ built in \cref{construcMP}, a finite tuple of objects $\overline{a}\in \mathcal{M}$ and a finite set of types $\pi_{i}(\overline{x})$ as in the first point of the statement. We can build a structure $A\in K$ containing $\overline{a}$ and with parameters $(p_{i})_{i\leq n}$ such that $\models \pi_{i}(\overline{a}_{p_{i}})$ for every $i\leq n$ and $\models q((p_{i})_{i\leq n})$. Then by saturation we can embed this structure $A$ inside of $\mathcal{M}$ over $\overline{a}$ and find the parameters.

\vspace{10pt}
The second point is similar: Consider the structure $B$ generated by $\overline{b},\overline{p}$ and types $\pi_{i}(\overline{x},\overline{b}_{p_{i}})$ as in the statement. We can extend $B$ into $B'\in K$ that contains a tuple of objects $\overline{a}'$ such that $\overline{a}'\overline{b}$ has the right type for the parameter $p_{i}$ for every $i$ and then embed $B'$ into $\mathcal{M}$ over $B$ to find the right objects.\end{proof}

\begin{definition} A formula $\varphi(x,y)$ has the \emph{independence property}, I.\ P.\ for short, if there is $(a_{i})_{i<\omega}$ and $(b_{I})_{I\subseteq \omega}$ two sequences of tuples such that $\models \varphi(a_{i},b_{I})$ if and only if $i\in I$. A theory $T$ has I.\ P.\ if some formula in it has I.\ P..
\end{definition}

\begin{cor}
The theory $T$ has I.\ P.\ if there are two distinct non-algebraic types in $T_{0}$. $T$ is non-simple if $T_{0}$ is non-trivial, in the sense that there are $a,b\in \mathbb{M}_{0}$ such that $\langle a,b\rangle \setminus (\langle a \rangle \cup \langle b \rangle) \centernot= \emptyset$.
\end{cor}

\begin{proof}
For I.P. consider two distinct $T_{0}$-types $\pi_{i}$ in $\overline{x}$ for $i=0,1$ and a formula $\psi(\overline{x})$ such that $\pi_{0}(\overline{x})\vdash \psi(\overline{x})$ and $\pi_{i}(\overline{x})\vdash \neg\psi(\overline{x})$. If we consider a sequence $(p_{i})_{i<\omega}$ of distinct parameters in $\mathbb{M}$, then by \ref{qecriterP} for every $I\subseteq \omega$ we can find $\overline{a}_{I}$ such that $\models \pi_{0,p_{i}}( \overline{a}_{I})$ if $i\in I$ and $\models \pi_{1,p_{i}}( \overline{a}_{I})$  if $i\centernot\in I$. Then the formula $\psi_{y}(\overline{x})$ has I.P..

\vspace{10pt}

Now for simplicity let $t(x,y)$ be a $L_{0}$-term such that $t(a,b)\centernot\in \langle a\rangle \cup \langle b \rangle$ and consider a tree of parameters all distinct $(p_{\mu})_{\nu \in \omega^{<\omega}}$. Lets now consider a tree of distinct objects $(a_{\mu})_{\mu \in \omega^{<\omega}}$ and the formulas $\varphi(x,y,a_{\mu\frown i},p_{\mu}):=t_{p_{\mu}}(x,y)=a_{\mu\frown i}$ along this tree. Clearly $\lbrace \varphi(x,y,a_{\mu\frown i},p_{\mu})$ : $i <\omega\rbrace$ is 2-inconsistent for every $\mu \in \omega^{<\omega}$, and by \ref{qecriterP} every path is consistent.
\end{proof}

\begin{remark} Even when the relation $\forkindep^{\Gamma}$ satisfies local character in $\mathbb{M}_{P}$ and $\mathbb{M}_{0}$ it does not need to do so in $\mathbb{M}$ (otherwise it would satisfy all conditions of the Kim-Pillay theorem): Take some distinct objects $\overline{o}=o_{1},o_{2},o_{3}$ and a set $P=(p_{i})_{i<\kappa}$ of parameters for $\kappa= \vert T_{P} \vert^{+}$ such that $o_{3}\in \langle o_{1}o_{2}\rangle_{p_{i}}$ for every $i<\kappa$ and $o_{j}\notin \langle P \rangle$ for every $j\in \lbrace 0,1,2\rbrace$. Then there is no subset $P_{0}$ of cardinality $\vert T \vert$ of $P$ such that $\overline{o}\forkindep^{\Gamma}_{P_{0}}P$.
\end{remark}

\begin{remark} In the case where we can find in $K_{0}$ some $c = t(a,b)\in \langle ab\rangle \setminus (\langle a \rangle \cup \langle b \rangle)$ the relation $\forkindep^{a}$ in $\mathbb{M}$ does not admit $4$-amalgamation: Take 3 distinct parameters $p_{0},p_{1},p_{2}$ and distinct objects $b_{0},b_{1},b_{2}$ such that $b_{i},b_{j}\centernot\in \langle p_{k} \rangle$ for every $i,j,k$ such that $\lbrace i,j,k\rbrace = \lbrace 0,1,2 \rbrace$ and such that $t_{p_{0}}(b_{1},b_{2})=t_{p_{1}}(b_{0},b_{2})$, $t_{p_{1}}(b_{0},b_{2})=t_{p_{2}}(b_{0},b_{1})$ and $t_{p_{0}}(b_{1},b_{2})\centernot=t_{p_{2}}(b_{1},b_{2})$. Then the type $\tp(p_{0}/b_{1}b_{2})\cup \tp(p_{1}/b_{0}b_{2})\cup \tp(p_{2}/b_{0}b_{1})$ is inconsistent.

\end{remark}


\section{Examples of generic extensions preserving $(H)$}

Let $K$ be a class of countable $L$-structures satisfying $(H)$. We describe now how adding generic structure to the class $K$ might affect the conditions $(H)$ and some properties of the theory.

\subsection{Adding a generic predicate}

\begin{definition}
Consider a new symbol $U$ for a unary predicate in some sort $S$, we define $L_{U}=L\cup\lbrace U \rbrace$ and the class $K_{U}$ of $L_{U}$-structures consisting of the structures $A$ in $K$ equipped with an arbitrary subset $U(A)\subseteq S(A)$. \end{definition}

\begin{prop}\label{genericrelation} The class $K_{U}$ satisfies $(H)$.
\end{prop}

\begin{proof}
The Condition $1\ $ is satisfied in $K_{U}$. For Condition $3\ $ we take $$U(A\oplus_{E}B):= U(A)\cup U(B),$$

where $A\oplus_{E}B$ is the free amalgam in $K$. Proving Condition $4\ $ is straightforward, we take $\langle \lbrace x \rbrace \rangle\in K$ and set $U$ to be empty. For Condition $5\ $ we find an amalgam $D$ using Condition $5\ $ in $K$ and set $U(D)$ to be $U(B_{0}) \cup U(B_{1})\cup U(B)$.

\vspace{10pt}
For Condition $2\ $ consider a finitely generated structure $\langle a,b\rangle \in K_{U}$ and a quantifier free $L_{U}$ formula $\varphi$ such that $\models \varphi(a,b)$. Without loss of generality we can assume that: 

$$ \varphi(x,y)=\varphi'(x,y)\wedge \bigwedge\limits_{i<n}U(t_{i}(x,y))\wedge \bigwedge\limits_{i<n}\neg U(t'_{i}(x,y)),$$ 

where $\varphi'(x,y)$ is an $L$-formula and the $t_{i}$ and $t'_{i}$ are $L$-terms. By taking a strengthening of $\varphi$ we can assume that $\varphi'(x,y) \vdash \bigwedge\limits_{j<n} t'_{i}(x,y)\centernot= t_{j}(x,y)$ for every $i<n$.

\vspace{10pt}
Now we consider the formula $\psi'(x)$ associated to the $L$-formula $\varphi'(x,y)\wedge \bigwedge\limits_{i,j<n} (t'_{i}(x,y)\centernot= t_{j}(x,y))$ by Condition $2\ $ in $K$. If $A'\models \psi(a')$ for some $a'\in A'\in K_{R}$ then by assumption we can find some $b'\in B'$ an extension of $A'$ in $K$ such that:

$$B' \models \varphi'(a',b')\wedge \bigwedge\limits_{i<n} (t'_{i}(a',b')\centernot\in \langle t_{j}(a',b'))_{j< n}\rangle).$$ 

Define the subset $U$ on $B'$ to be $ U(A')\cup\lbrace t_{j}(a',b'))_{j< n}\rbrace$, we get that $B'\models \neg U(t'_{i}(a',b'))$ for all $i<n$, so $B'\models\varphi(a',b')$, and $B'$ is an extension of $A'$ in $K_{U}$.\end{proof}

\begin{remark}
If there are some $a,b \in A\in K$ and a term $t(x,y)$ such that $ t(a,b)\in \langle ab\rangle \setminus (\langle a \rangle \cup \langle b\rangle)$ the limit theory $T_{U}$ of $K_{U}$ is non stable, in fact the formula $U(t(x,y))$ has the order property.\end{remark}

\begin{remark}\label{remarkrelatgen} The proof of \cref{genericrelation} work more generally for a relation of any arity: If we were working with an $n$-ary relation we would have some tuples of terms $\overline{t}_{i}$ and $\overline{t}'_{i}$ in the formula $\varphi$ from the beginning of the proof and the formula $\varphi'$ should express that $\overline{t}_{i}\centernot=\overline{t}'_{j}$ for any $i,j$.
\end{remark}

\begin{definition}
Let $L_{R}=L\cup\lbrace R \rbrace$ be a new unary predicate symbol. We define the class $K_{R}$ of $L_{R}$-structures consisting of a structure $A$ in $K$ equipped with an arbitrary substructure $R(A)$. \end{definition}

\begin{prop} If the structures of $K$ are uniformly locally finite the class $K_{R}$ satisfies Condition $1\ -4$. However if the generated structure in $K$ is non-trivial (i.e. in $\langle A \rangle \centernot= A$ for some set $A$) then Condition $5\ $ needs not to hold.
\end{prop}

\begin{proof}
The Conditions $1$, $3$ and $4$ are still true by the same arguments as before, we will need the local finiteness assumption on $K$ to get Condition $2$. We define $P(A\oplus_{E}B):= \langle P(A)P(B)\rangle $ and $P(\langle \lbrace x \rbrace \rangle):= P(\langle \emptyset \rangle)$. For Condition $2\ $ consider a finitely generated structure $\langle a,b\rangle \in K_{R}$ and a quantifier free $L_{R}$ formula $\varphi$ such that $\langle a,b\rangle\models \varphi(a,b)$. Without loss of generality we can assume that: 

$$ \varphi(x,y)=\varphi'(x,y)\wedge \bigwedge\limits_{i<n}R(t_{i}(x,y))\wedge \bigwedge\limits_{i<n}\neg R(t'_{i}(x,y)),$$ 

where $\varphi'(x,y)$ is an $L$-formula, for all $c\in R(\langle ab\rangle)$ there is some $i<n$ such that $c=t_{i}(a,b)$ and for all $c\in \langle ab\rangle \setminus R(\langle ab\rangle)$ there is some $i<n$ such that $c=t'_{i}(a,b)$.

\vspace{10pt}
Since our theory is uniformly locally finite the Condition $t'_{i}(x,y)\centernot\in \langle (t_{j}(x,y))_{j< n}\rangle$ is definable. Now we consider the formula $\psi'(x)$ associated to the $L$-formula $\varphi'(x,y)\wedge \bigwedge\limits_{i<n} (t'_{i}(x,y)\centernot\in \langle (t_{j})_{j< n}\rangle)$ by Condition $2\ $ in $K$, and then the formula:

$$\psi(x) :=\psi'(x)\wedge \bigwedge\limits_{k<m} R(t_{i_{k}}(x)) \wedge \bigwedge\limits_{k<m} \neg R(t_{i_{k}}(x)),$$

where $t_{i_{k}}(x)$ enumerates the terms $t$ such that $R(t(a))$ and $t'_{i_{k}}(x)$ enumerates the terms $t'$ such that $\neg R(t'(a))$. If $A'\models \psi(a')$ for some $a'\in A'\in K_{R}$ then by assumption we can find some $b'$ such that $\langle a'b'\rangle \models \varphi'(a',b')\wedge \bigwedge\limits_{i<n} (t'_{i}(a',b')\centernot\in \langle t_{j}(a',b'))_{j< n}\rangle)$. Define the substructure $R$ on $\langle a'b'\rangle$ to be $ \langle t_{j}(a',b'))_{j< n}\rangle$. We obtain that $\langle a'b'\rangle\models \neg R(t'_{i}(a',b'))$ for all $i<n$, so $\langle a'b'\rangle\models\varphi(a',b')$, and this property still holds in $B':=A'\oplus_{\langle a' \rangle}\langle a'b'\rangle$, which is an extension of $A$ as an $L_{R}$-structure.

\vspace{10pt}

Now about Condition $5\ $: By Condition $5\ $ in the class $K$ we can find some structure $D\in K$ and some embeddings as in the statement. We want to define $R(D)$. We know that in any case $\langle R(D_{0})\cup R(D_{1}) \cup R(B)\rangle \subseteq R(D)$, however it might happen that $b\in \langle R(D_{0})\cup R(D_{1}) \cup R(B)\rangle $ for some $b\in B\setminus R(B)$, and in that case the $L$-embedding $\varphi_{B}$ can not be an $L_{R}$-embedding.\end{proof}

To give a simple example of this last fact, let us look at some vector space over a finite field with a predicate for a subspace. Take some independent vectors $a_{0},a_{1},b_{0},b_{1}$ as in the statement of the independence theorem for $\forkindep^{a}$ such that $\neg R(a_{0}+b_{0})$, $ R(a_{1}-b_{1})$ and $R(b_{0}+b_{1})$ are satisfied. Then $R(a+b_{0})$ is satisfied if $a$ satisfies $\tp(a_{i}/b_{i})$ for $i=0,1$, which is a contradiction. In that case the limit of $K_{R}$ exists, and it might be NSOP$_1$, however $\forkindep^{K}\centernot= \forkindep^{a}$ in the limit theory.

\begin{remark} If we assume that the class $K$ is uniformly locally finite we can restrict it to the class of finite structures in $K$ and have a Fraïssé class in the usual sense, then its limit is $\omega$-categorical.
\end{remark}

As we have seen in the previous section Condition $5\ $ puts a strong restriction on what we can add to the structure, and some classes with an NSOP$_1$ limit can satisfy the Conditions $1\ -4\ $ and not satisfy Condition $5$. For example the class of finite equivalence relations, with limit the (stable) theory of the equivalence relation with an infinite number of classes all infinite (which we will write $T_{\infty}$), or its parameterized version (see \cite{kaplan2020kim} or \cite{bossut2023note}) which is strictly NSOP$_1$, satisfies existence and that $\forkindep^{K^{M}}=\forkindep^{f}=\forkindep^{d}$ over arbitrary sets.

\begin{definition}
Consider a new symbol $E$ for a binary predicate in some sort $S$, we define the language $L_{E}=L\cup\lbrace V,E,p \rbrace$ where $V$ is a new sort, $E$ a binary relation on $S$, and $p$ a unary map from $S$ to $V$. Define the class $K_{E}$ of $L_{E}$-structures consisting of any structure $A$ in $K$ equipped with an equivalence relation $E^{A}\subseteq S(A)^{2}$, a new sort $V(A)$ and a map $p : S(A)\rightarrow V(A)$ such that $p(a)=p(b)$ if and only if $E(a,b)$ for all $a,b\in S(A)$.\end{definition}

\begin{prop} The class $K_{E}$ satisfies $(H)$.
\end{prop}

\begin{proof} If we do not add the quotient to the structure Conditions $1\ -4\ $ are satisfied but Condition $5\ $ is not, also if we ask for the map $p$ to be surjective the Condition $4\ $ will not hold. It is easy to see that Conditions $1\ -4\ $ are satisfied in $K_{E}$: The only structure on the quotient is equality and we can do `in parallel' free amalgamation in $K$ and in the quotient. For $5\ $ we proceed similarly by forming the amalgam in $K$ on one end and the amalgam of the quotient on the other end.\end{proof}

Notice that this construction coincides with adding a new sort $V$ with no structure other than equality and then a generic function from $S$ to $V$.

\subsection{Adding generic functions}

\begin{definition}
Consider a new symbol $f$ for a binary function, we define $L_{f}=L\cup\lbrace f \rbrace$ and the class $K_{f}$ of $L_{f}$-structures consisting of a structure $A$ in $K$ equipped with a binary function $f_{A} : A\times A\rightarrow A$ such that the structure generated by the constants is some specified $L_{f}$-structure (in order to have the joint embedding property and uniqueness of the limit the structures need to agree on the constants). \end{definition}

\begin{prop}\label{genericfonction} The class $K_{f}$ satisfies $(H)$.
\end{prop}

\begin{proof}
Condition $1\ $ is clearly satisfied in $K_{f}$. For Condition $3\ $ we construct some $L_{f}$-structure $D$ using $3\ $ and $4\ $ in $K$:

\vspace{10pt}
Consider $E\subseteq A,B$ two extensions of $L_{f}$-structures. We define an increasing sequence of $L$-structures $(D_{i})_{i<\omega}$ equipped with a partial binary function $f_{i+1}: D_{i}^{2}\rightarrow D_{i+1}$. Let $D_{0}:=A\oplus_{E}B$ and $f_{0}=f_{E}$. Now set: 

$$D_{1}:= D_{0}\oplus_{\langle \emptyset \rangle}\langle \lbrace D_{0}^{2}\setminus (A^{2}\cup B^{2})\rbrace \rangle,$$

and consider $f_{1} : D_{0}^{2} \rightarrow D_{1}$ defined the following way:

\begin{enumerate}
    \item $f_{1}(x,y):=f_{A}(x,y)$ if $(x,y)\in A^{2}$,
    \item $f_{1}(x,y):=f_{B}(x,y)$ if $(x,y)\in B^{2}$,
    \item $(x,y)$ is sent to the element of $D_{1}$ with the corresponding label otherwise.
\end{enumerate}

We iterate this process: if $D_{i},f_{i+1}$ are defined for $i>0$ we let:

$$D_{i+1}:= D_{i}\oplus_{\langle \emptyset \rangle}\langle \lbrace
D_{i}^{2}\setminus (D_{i-1}^{2})\rbrace \rangle,$$

and $f_{i+1}$ extends $f_{i}$ by sending $(x,y)\in D_{i}^{2}\setminus (D_{i-1}^{2})$ to the element of $D_{i+1}$ with the corresponding label. We then set $D:=\bigcup\limits_{i<\omega}D_{i}$ and $f:=\bigcup\limits_{i<\omega}f_{i} : D^{2}\rightarrow D$. This structure satisfies the universal property: it is generated by the images of $A$ and $B$ as an $L_{f}$-structure, given $D'$ and embeddings $\varphi_{A} : A \rightarrow D'$, $\varphi_{B} : B \rightarrow D'$ as in the statement we can define an increasing sequence of morphisms of $L$-structures $\varphi_{i} : D_{i} \rightarrow D'$ in a canonical way and show that its union is a morphism of $L_{f}$-structures. The construction of $\langle \lbrace x \rbrace \rangle$ in $K_{f}$ follows the same method. 

\vspace{10pt}
Condition $5\ $ is easy to check: Consider $E,A,B_{i},D_{i}$ for $i=0,1$ and $D$ as in the statement. We apply $5\ $ inside of $K_{0}$ to find embeddings of $L$-structures into some $D\in K$, and define the function $f_{D}$ in $D$ so that these maps become embeddings of $L_{f}$-structures: Take $f_{D}\vert_{D_{i}^{2}}$ to be the image of $f_{D_{i}}$ by $\varphi_{D_{i}}$, similarly for $B$, and we can set $f$ to be anything outside of $D^{2}_{i}$ and $B^{2}$.

\vspace{10pt}
For Condition $2\ $ consider a structure $\langle ab\rangle \in K_{f}$ and some quantifier-free $L_{f}$-formula $\varphi$ such that $\langle ab\rangle \models \varphi(a,b)$. By replacing every instance of $f$ in $\varphi$ by a new variable and separating the instances of $f$ applied only to elements of $\langle a \rangle$ from the other ones we can reshape $\varphi(x,y)$ into a formula of the following form:

$$\varphi'(x_{\leq m},y_{\leq n})\wedge \bigwedge\limits_{i\leq m}(x_{i}=f(s_{i}(x_{<i}),s'_{i}(x_{<i})))\wedge \bigwedge\limits_{i\leq n}(y_{i}=f(t_{i}(x_{\leq m},y_{<i}),t'_{i}(x_{\leq m},y_{<i})))$$
where $\varphi'$ is an $L$ quantifier-free formula and the $t_{i},s_{i},t'_{i},s'_{i}$ are $L$-terms. We define by induction $x_{0}(a)=a$, $x_{k}(a)= f(s_{k}(x_{<k}(a)),s'_{k}(x_{<k}(a)))$ if $k>0$, $y_{0}(a,b)=b$ and if $k>0$ $$y_{k}(a,b):=f(t_{k}(x_{\leq m}(a),y_{<k}(a,b)),t'_{k}(x_{\leq m}(a),y_{<k}(a,b))).$$ If $y_{k}(a,b)\centernot=y_{k'}(a,b),$ we strengthen the formula $\varphi$ to specify that: $$(t_{k}(x_{\leq m},y_{<k}),t'_{k}(x_{\leq m},y_{<k})) \centernot= (t_{k'}(x_{\leq m},y_{<k'}),t'_{k'}(x_{\leq m},y_{<k'})),$$ and similarly for $(s_{k}(x_{\leq k}),s'_{k}(x_{\leq k}))$ and $(s_{k'}(x_{\leq k'}),s'_{k'}(x_{\leq k'}))$.

\vspace{10pt}
Now consider the $L$-formula $\psi'(x_{\leq m})$ associated by Condition $2\ $ in $K$ to  $\varphi'(x_{\leq m},y_{\leq n})$ and then the $L_{f}$-formula:

$$\psi(x):=\psi'(x_{\leq m}) \wedge \bigwedge\limits_{i\leq m}(x_{i}=f(s_{i}(x_{<i}),s'_{i}(x_{<i}))).$$

Assume that $A'\models \psi (a')$ for some $a'\in A'$. By assumption there is an $L$-extension $A'\subseteq B'$ and $b'_{\leq n} \in B'$ such that:

$$B'\models \varphi'(a'_{\leq m},b'_{\leq n})\wedge \bigwedge\limits_{i\leq m}(a'_{i}=f(s_{i}(a'_{<i}),s'_{i}(a'_{<i}))).$$

We just need to define a binary function $f_{B'}$ on $B'$ extending $f_{A'}$ such that

$$b'_{k}=f(t_{i}(a'_{\leq m},b'_{<i}),t'_{i}(a'_{\leq m},b'_{<i}))$$ 

for every $k\leq n$. We begin by setting these values, this is possible thanks to the strengthening we had, and we can let $f_{B'}$ take any values elsewhere. We write $B'_{f}$ for the $L_{f}$-structure defined on $B'$, then $B'_{f}\models \varphi(a',b')$.\end{proof}

\begin{remark} Adding a generic binary function assures us that the theory of any limit is strictly NSOP$_1$, as proven in \cite[Section 3]{kruckman2018generic}. Also we get that the generated structure is non-trivial. With the same construction we can add generic functions of any arity.
\end{remark}

\begin{remark}\label{expansgeneric} The proof of \cref{genericfonction} work more generally for a function of any arity: Adding variables to the generic function just makes the $L'$-formulas $\psi'$ longer to write down but the proof works similarly. This remark added with \cref{remarkrelatgen} gives us that if some class of structures $K$ satisfy $(H)$ then the expansion of $K$ by some generic language also does.
\end{remark}

\begin{definition} We define $L_{\pi}=L\cup \lbrace \pi, \pi^{-1} \rbrace$ new symbols of unary functions and the class $K_{\pi}$ of $L_{\pi}$-structures consisting of a structure $A$ in $K$ equipped with a bijection $\pi_{A}$ of inverse function $\pi^{-1}_{A}$ such that the structure generated by the constants is some specified $L_{\pi}$-structure.
\end{definition}

\begin{prop}The class $K_{\pi}$ satisfies $(H)$.
\end{prop}

\begin{proof}
Condition $1\ $ is clearly satisfied in $K_{\pi}$, for the Conditions $2\ ,4\ $ and $5\ $ we can proceed as previously. Now we describe how to build the free amalgam for Condition $3\ $. Let $E\subseteq A,B$ be two extensions in $K_{\pi}$.
\vspace{10pt}

In the same fashion as previously we construct an increasing sequence of $L$-structures $(D_{i})_{i<\omega}$, and an increasing sequence of subsets $H_{i}\subseteq D_{i}$ and two increasing sequences of functions $\pi_{i},\pi_{i}^{-1} : H_{i}\rightarrow D_{i}$ such that $D_{i}\subseteq H_{i+1}$ and $\pi_{i+1}\circ \pi^{-1}_{i}=\pi_{i+1}^{-1}\circ \pi_{i}=id_{H_{i}}$ for every $i<\omega$. Let $D_{0}:=A\oplus_{E}B$, $H_{0}=A\cup B$, and $\pi_{0}=\pi_{A}\cup \pi_{B}$, $\pi^{-1}_{0}=\pi^{-1}_{A}\cup \pi^{-1}_{B}$. 

$$D_{1}:=D_{0}\oplus_{\langle \emptyset \rangle}\langle \lbrace [D_{0}\setminus (A\cup B)]_{k}\ :\ k \in \mathbb{Z}\setminus \lbrace 0 \rbrace\rbrace \rangle,$$ 

I.e., we consider the $L$-structure freely generated over $D_{0}$ by copies of $D_{0}\setminus (A\cup B)$ indexed on $\mathbb{Z}\setminus \lbrace 0 \rbrace$, we see the element of $D_{0}\setminus (A\cup B)$ as those corresponding to the copy of index $0$. Set $$H_{1}=D_{0}\cup\bigcup\limits_{k\in \mathbb{Z}\setminus \lbrace 0 \rbrace}[D_{0}\setminus (A\cup B)]_{k}.$$

We define $\pi_{1}$ as the extension of $\pi_{0}$ that sends an element of $[D_{0}\setminus (A\cup B)]_{k}$ to its copy of index $k+1$, i.\ e.\  it sends $t(a,b)_{k}$ to $t(a,b)_{k+1}$, similarly for $\pi_{0}^{-1}$ which sends an element of $[D_{0}\setminus (A\cup B)]_{k}$ to its copy of index $k-1$. Assume that $D_{i+1}$,$H_{i}$ and $\pi_{i},\pi_{i}^{-1}$ have been defined for $i>0$. We set

$$D_{i+1}:= D_{i}\oplus_{\langle \emptyset \rangle}\langle \lbrace [D_{i}\setminus D_{i-1}]_{k}\ :\ k \in \mathbb{Z}\setminus \lbrace 0 \rbrace\rbrace \rangle,$$ 

$$H_{i+1}:= D_{i}\cup \bigcup\limits_{k \in \mathbb{Z}\setminus \lbrace 0 \rbrace}[D_{i}\setminus D_{i-1}]_{k},$$

and take $\pi_{i+1}$ to be the function extending $\pi_{i}$ that sends an element of $[D_{i}\setminus D_{i-1}]_{k}$ to its copy of index $k+1$, similarly for $\pi_{0}^{-1}$ which sends an element of $[D_{i}\setminus D_{i-1}]_{k}$ to its copy of index $k-1$. Once this is done we take $D:=\bigcup\limits_{i<\omega} D_{i}$, $\pi_{D}:=\bigcup\limits_{i<\omega} \pi_{i}$ and $\pi^{-1}_{D}:=\bigcup\limits_{i<\omega} \pi^{-1}_{i}$.

\vspace{10pt}
This defines an $L_{\pi}$-structure of bijection $\pi_{D}$, and $D$ satisfies the universal property: it is generated by the images of $A$ and $B$ as an $L_{\pi}$-structure, given $D'$ and embeddings $\varphi_{A} : A \rightarrow D'$, $\varphi_{B} : B \rightarrow D'$ as in the statement we can define an increasing sequence of morphism of $L$-structures $\varphi_{i} : D_{i} \rightarrow D'$ in a canonical way and show that its union is a morphism of $L_{\pi}$-structures. We take $\varphi_{i} : D_{i} \rightarrow D'$ to be the only morphism of $L$-structures such that $\varphi_{i}\vert_{A}=\varphi_{A}$ and $\varphi_{i}\vert_{B}=\varphi_{B}$. If $\varphi_{i} : D_{i} \rightarrow D'$ is defined we take $\varphi_{i+1} : D_{i+1} \rightarrow D'$ to be its unique morphism extending $\varphi_{i}$ that sends $\pi_{D}^{k}(t)$ to $\pi^{k}_{D'}(\varphi(t))$ for every $t\in D_{i}$.

\vspace{10pt}
We define $\varphi:=\bigcup\limits_{i<\omega} \varphi_{i}$, it is a morphism of $L$-structures and by construction it commutes with $\pi$, so it is a morphism of $L_{\pi}$-structures. The construction of $\langle\lbrace x \rbrace \rangle$ in $K_{\pi}$ follows the same method.\end{proof}

\begin{definition}
We define $L_{\sigma}=L\cup \lbrace \sigma, \sigma^{-1} \rbrace$ new symbols of unary functions and the class $K_{\sigma}$ of $L_{\sigma}$-structures consisting of a structure $A$ in $K$ equipped with an automorphism $\sigma_{A}$ and its inverse function $\sigma^{-1}_{A}$.
\end{definition}

In $K_{\sigma}$ the Conditions $1\ -4\ $ are easy to check: we use the Free Amalgam property to show that if $E\subseteq A,B \in K_{\sigma}$ then there is an (unique) automorphism $\sigma_{A\oplus_{E}B}$ of $A\oplus_{E}B$ extending $\sigma_{A}$ and $\sigma_{B}$ and $(A\oplus_{E}B,\sigma_{A\oplus_{E}B})$ satisfies the universal property.

\vspace{10pt}
For Condition $4\ $ we consider the $L$-structure $\langle \lbrace (x_{i})_{i\in \mathbb{Z}}\rbrace \rangle \in K$ equipped with the unique automorphism that extends $\sigma_{A}$ and sends $x_{k}$ to $x_{k+1}$ for every $k\in \mathbb{Z}$. As for the generic substructure we need the structures of $K$ to be uniformly locally finite for Condition $2\ $ to hold, in this case the proof follows the same method. Also Condition  $5\ $ need not to hold even in this case for similar reasons as in $K_{R}$. The limit of $K_{\sigma}$ exists, and it might be NSOP$_1$, however in that case $\forkindep^{K}\centernot= \forkindep^{a}$ in the limit theory.

\begin{remark}
To sum up the last two sections, the theories we are building are model companions of theories which are made from very basic one-based stable theories, like the theory of the $\mathbb{F}_{2}$-vector space of infinite dimension, to which we add some basic structure (equivalence relation, bijection,..) that does not interact - typically some generic functions and relations as in \cref{expansgeneric}. This is slightly more general than the model companion of the empty theory, and it allows us to give some minimalist examples of strictly NSOP$_1$ theory: The model companion of the theory of an abelian group with a unary function $f$ such that $f(0)=0$.\end{remark}

\section{Some perspectives}

The theory $T$ of the limit of some $K$ satisfying $(H)$ is not a free amalgamation theory in the sense of \cite{conant2017axiomatic}. The stationary independence relation $\forkindep^{\Gamma}$ does not in general satisfies the closure condition (which is defined as : if $C\subseteq A,B$ are algebraically closed and $A\forkindep_{C}B$ then $A\cup B$ is algebraically closed), however the relation $\forkindep^{\Gamma}$ is interesting.

\vspace{10pt}
There are other known examples of NSOP$_1$ theories admitting a strong independence relation satisfying \cref{gammapropP}, which could motivate the study of such theories :
\begin{enumerate}
    \item The theory of $\omega$-free PAC fields of characteristic $0$ (see \cite{Chatzidakis2008IndependenceI} and \cite{Chatzidakis2002PropertiesOF}) with the the relation $\forkindep^{III}$ defined as: for $C\subseteq A,B$, $A\forkindep^{III}_{C}B$ if $A\forkindep^{lin}_{C}B$ and $acl(AB)=AB$. This relation is named strong independence in \cite[Definition 4.5]{fieldssimplicity}, written as here in \cite[Section 3.6]{Chatzidakis2008IndependenceI}, but this definition is not the same as the strong independence of \cite[Definition 1.2]{Chatzidakis2002PropertiesOF}.
    \item The theory of vector spaces of infinite dimension over an ACF with a non degenerate bilinear form with the relation $\forkindep^{\Gamma}$ (see \cite{adler2009geometric}, \cite{dobrowolski2020sets} or \cite{bossut2023note} for the definition of $\forkindep^{\Gamma}$). This is not a free amalgamation theory because of the closure condition.
    \item The theory ACFG of an ACF with a generic multiplicative subgroup, see strong independence in \cite[Section 2.2]{d2021forking} and the remark following \cite[Section 4.4]{d2021forking}.
\end{enumerate}

An other context that might be interesting to study is the one of `one-based' NSOP$_1$ theory. By this we mean an NSOP$_1$ theory with existence and weak elimination of imaginaries such that Kim-independence coincides with algebraic independence over arbitrary sets.

\begin{enumerate}
    \item If $T$ is a one-based NSOP$_1$ theory does forking independence coincides with $M$-independence?
    \item If $T$ is an NSOP$_1$ theory such that Kim-independence coincides with algebraic independence over models does $T$ satisfies existence? If that is the case does $\forkindep^{K}$ coincides with $\forkindep^{a}$ over arbitrary sets?
\end{enumerate}

\bibliographystyle{plain}
\bibliography{ref.bib}

\end{document}